\documentclass[12pt,reqno]{amsart}
\usepackage[margin=1in]{geometry}
\usepackage{amsmath,amssymb,amsthm,graphicx,amsxtra, setspace}
\usepackage[utf8]{inputenc}
\usepackage{mathrsfs}
\usepackage{hyperref}
\usepackage{upgreek}
\usepackage{mathtools}
\usepackage[dvipsnames]{xcolor}
\usepackage[mathcal]{euscript}
\allowdisplaybreaks

\DeclareMathOperator*{\esssup}{ess\,sup}
\DeclareMathOperator*{\essinf}{ess\,inf}

\usepackage[pagewise]{lineno}

\DeclareMathAlphabet{\mathpzc}{OT1}{pzc}{m}{it}

\usepackage[cyr]{aeguill}

\colorlet{darkblue}{blue!50!black}

\hypersetup{
	colorlinks,%
	citecolor=blue,%
	filecolor=red,%
	linkcolor=red,%
	urlcolor=blue,%
	pdfnewwindow=true,%
	pdfstartview={FitH}
}

\newtheorem{theorem}{Theorem}[section]
\newtheorem{lemma}[theorem]{Lemma}
\newtheorem{proposition}[theorem]{Proposition}

\newtheorem{definition}[theorem]{Definition}
\newtheorem{problem}[theorem]{Problem}
\newtheorem{example}[theorem]{Example}
\newtheorem{remark}[theorem]{Remark}

\newtheorem{hypothesis}[theorem]{Hypothesis}

\allowdisplaybreaks

\let\originalleft\left
\let\originalright\right
\renewcommand{\left}{\mathopen{}\mathclose\bgroup\originalleft}
\renewcommand{\right}{\aftergroup\egroup\originalright}

\renewcommand{\d}{\/\mathrm{d}\/}

\def\w{\textbf{W}^{\varepsilon}_{{\theta}^{\varepsilon}}}

\def\L{\mathbb{L}}
\def\A{\mathscr{A}}
\def\B{\mathscr{B}}
\def\C{\mathscr{C}}

\def\f{\boldsymbol{f}}

\def\q{\mathfrak{q}}

\def\E{\mathbb{E}}
\def\X{\mathbb{X}}

\def\g{\boldsymbol{g}}
\def\p{\mathfrak{p}}

\def\z{\boldsymbol{z}}
\def\v{\mathfrak{v}}
\def\w{\boldsymbol{w}}

\def\N{\mathbb{N}}
\def\r{\mathfrak{r}}

\def\V{\mathscr{V}}
\def\wi{\widetilde}

\def\u{\boldsymbol{y}}
\def\H{\mathscr{H}}
\def\n{\boldsymbol{n}}

\newcommand{\eps}{\varepsilon}

\newcommand{\R}{\mathbb{R}}

\renewcommand{\d}{\/\mathrm{d}\/}

\newcommand{\Addresses}{{% additional braces for segregating \footnotesize
		\footnote{
			\noindent \textsuperscript{1,2,3}Department of Mathematics, Indian Institute of Technology Roorkee-IIT Roorkee,
			Haridwar Highway, Roorkee, Uttarakhand 247667, INDIA.\par\nopagebreak
			\noindent  \textit{e-mail:} \texttt{Manil T. Mohan: maniltmohan@ma.iitr.ac.in, maniltmohan@gmail.com.}
			
			\textit{e-mail:} \texttt{Jyoti Jindal: jyoti@ma.iitr.ac.in.}
			
			\textit{e-mail:} \texttt{Sagar Gautam: sagar\_g@ma.iitr.ac.in.}
			
			\noindent \textsuperscript{*}Corresponding author.
			
			\textit{Key words:} convective Brinkman-Forchheimer extdened Darcy equations, Hemivariational inequality, Galerkin method, monotoncity.
			
			Mathematics Subject Classification (2020): Primary 47J20,  49J52, 35R70; Secondary  35Q35, 76D03.
		}}}

\begin{document}
	%	\linenumbers
	
	\title[Hemivariational inequality for 2D and 3D CBFeD equations]{A domain hemivariational inequality for 2D and 3D convective Brinkman-Forchheimer extended Darcy equations
		\Addresses}
	\author[J. Jindal, S. Gautam and M. T. Mohan]
	{Jyoti Jindal\textsuperscript{1}, Sagar Gautam\textsuperscript{2}, and Manil T. Mohan\textsuperscript{3*}}
	
	\maketitle

		\begin{abstract}
This paper investigates domain hemivariational inequality problems arising from the non-stationary two- and three-dimensional convective Brinkman-Forchheimer extended Darcy (CBFeD) equations, which describe the flow of viscous incompressible fluids through saturated porous media in bounded domains. These equations may be regarded as generalized Navier-Stokes systems incorporating both damping and pumping mechanisms. For all admissible absorption exponents $r \ge 1 $ and effective viscosity $\mu > 0 $, the existence of weak solutions to the non-stationary 2D and 3D CBFeD equations with hemivariational inequalities is established via a regularized Galerkin approximation scheme, based on a suitable regularization of the Clarke subdifferential. A noteworthy aspect of the analysis is that the existence results extend to the three-dimensional non-stationary Navier-Stokes equations. Moreover, under appropriate conditions on the absorption exponent, specifically, $r \ge 1 $ in two dimensions and $ r \ge 3 $ in three dimensions, it is shown that weak solutions satisfy the energy equality. In addition, uniqueness of solutions is proved for $ r \ge 1$ in 2D and $r \ge 3$ in 3D, with the additional requirement $2\beta \mu > 1 $ in the critical case $r = 3 $.
\end{abstract}
	
\section{Introduction}\setcounter{equation}{0}
	\subsection{The model} 
In this manuscript, we study the 2D and 3D CBFeD equations subject to a domain hemivariational inequality. We consider the following initial-boundary value problem for the 2D and 3D CBFeD system on bounded domain $\mathfrak{D} \subset \mathbb{R}^d$ for $d\in\{2,3\}$, which describes the steady flow of an incompressible viscous fluid through a porous medium (see \cite{MTT}):
\begin{equation}\label{eqn-dom}
	\left\{
	\begin{aligned}
		\frac{\partial \u}{\partial t}-\mu \Delta\u+(\u\cdot\nabla)\u+\alpha|\u|^{q-1}\u+\beta|\u|^{r-1}\u+\nabla \pi&=\f+\boldsymbol{g}, \ \text{ in } \ \mathfrak{D}\times (0,T), \\ \nabla\cdot\u&=0, \ \text{ in } \ \mathfrak{D}\times[0,T), \\
		\u&=\mathbf{0}\ \text{ on } \ \partial\mathfrak{D}\times[0,T], \\
		\u(0)&=\u^0 \ \text{ in } \ \mathfrak{D},
	\end{aligned}
	\right.
\end{equation}
where $\u(\cdot,\cdot) : \mathfrak{D} \times [0,T] \to \mathbb{R}^d$ denotes the flow velocity field, $\pi(\cdot,\cdot) : \mathfrak{D} \times [0,T] \to \mathbb{R}$ is the pressure and $\f(\cdot,\cdot) : \mathfrak{D} \times [0,T] \to \mathbb{R}^d$ is given forcing term. In addition, we are assuming
that the forcing term $\g(\cdot,\cdot):\mathfrak{D} \times [0,T]\to\mathbb{R}^d$  satisfies
\begin{align}\label{eqn-sub-diff}
	-\boldsymbol{g}(x,t)\in \partial_{\mathpzc{C}} j(x,t,\u(x, t))\ 
	\text{ in }\ \mathfrak{D}\times(0,T),
\end{align}
where $\partial_{\mathpzc{C}} j$ is the subdifferential of the superpotential 
$j:\mathfrak{D}\times(0,T)\times\mathbb{R}^d\to\mathbb{R}$ which is defined in the sense of Clarke (see Definition \ref{cLarkedef}). 
In the above model \eqref{eqn-dom}, the constants $\mu, \beta>0$ denotes the \emph{Brinkman coefficient} (effective viscosity) and  \emph{Forchheimer coefficient} (proportional to the porosity of the material), respectively. In addition, the term $\beta|\u|^{r-1}\u$ is known as the absorption or damping term with absorption exponent $r\geq1$ (cf. \cite{SNAHB}). In particular, $r=3$ is known as the \emph{critical exponent}. The cases $r<3$ and $r>3$ are known as \emph{subcritical} and \emph{supercritical}, respectively. 

As mentioned in \cite{MTT}, the additional nonlinear term $\alpha |\u|^{q-1}\u$ introduced by the authors in the system \eqref{eqn-dom} which represents a potential pumping mechanism that counteracts the damping effect when $\alpha < 0$, as assumed throughout this work. The parameter $ q \in [1, r)$ governs the strength of the pumping effect. For $\alpha=\beta=0$, the system \eqref{eqn-dom} reduces to the classical Navier-Stokes equations (abbreviated as NSE), and if $\alpha, \beta>0$, then it can be considered as damped NSE. The above model \eqref{eqn-dom} is accurate when the flow velocity is too large for Darcy's law to be valid as long as the porosity is not too small (see \cite{MTT}). For physical significance of the model, the interested readers are referred to see \cite{SGKKMTM,MTT}. 
\subsection{Literature survey}
In the literature, convective Brinkman-Forchheimer equations (abbreviated as CBF) are also known as tamed NSE or NSE modified with an absorption term, cf. \cite{SNAHB,MRXZ} etc., and references therein. The damping $\beta|\u|^{r-1}\u$ arises from the resistance to the motion of the flow, which describes several physical phenomena such as drag or friction effects, porous media flow, some dissipative mechanisms, cf. \cite{ZCQJ,KWH,MTT,ZZXW} etc., and references therein. The global solvability results for the CBF model involving fast growing nonlinearities in both two- and three-dimensions can be found in \cite{SNAHB,SGMTM,KWH,KT2,MTT}.

We are naturally led to a mathematical model that involves the Clarke subdifferential of a locally Lipschitz superpotential. This formulation, known as a hemivariational inequality, was first introduced and studied in the early 1980s by Panagiotopoulos \cite{PDP, PDPHI}. Hemivariational inequalities serve as natural extensions of variational inequality problems and originate from non-convex and non-smooth mechanics and one can see \cite{WH1} for more information regarding variational-hemivariational inequalities.  In \cite{WHYYSZ}, the authors analyzed a time-dependent mixed hemivariational inequality arising from incompressible fluid flow governed by the Stokes system. The model incorporates a nonsmooth friction-type boundary condition formulated using the Clarke subdifferential.  The authors in \cite{Fang2016, PKa} studied the existence of solutions of the hemivariational inequalities for the NSE, through a unified framework of an abstract problem and explores the
uniqueness, and continuous dependence on the data. In addition, they applied the results to the nonstationary
hemivariational inequalities for the NSE that are of the boundary type, corresponding to nonlinear slip boundary conditions, and of the domain type, corresponding to hydraulic flow controls.   As compared to boundary type inequalities, a domain hemivariational inequality is a mathematical problem in which the non-smooth and non-convex superpotential influences the system throughout the interior of the spatial region
$\mathfrak{D}$, instead of being restricted only to the boundary.
The authors in \cite{Fang2016, PKa} studied the existence  of solution through constructing a temporally semi-discrete approximation whose solutions converge to a solution of the Navier-Stokes hemivariational inequalities (NSHVIs). This approach has been successfully applied, for instance, in the analysis of the time-dependent hemivariational inequalities with CBF equations that are of the boundary type, where the authors of \cite{JSMT} proved both the existence and uniqueness of weak solutions.

\subsection{Comparison with the existing literature}
Most studies on hemivariational inequalities (HVIs) establish the existence of solutions either through fixed-point techniques or by employing Rothe’s method. An alternative and widely used approach is the Galerkin method combined with a suitable regularization procedure, as applied in \cite{SMAON} to study boundary type HVIs arising in the NSE. In the Galerkin discretization scheme, the infinite-dimensional operator equation is replaced with a finite-dimensional one, whose well-posedness can be shown via classical existence and uniqueness results or fixed-point arguments. After obtaining uniform a-priori estimates for the approximate solutions, the Banach-Alaoglu theorem is used to extract a weakly convergent subsequence. By applying some compactness results like Aubin-Lions, one can extract a further subsequence which converges strongly, so that the weak limit serves as a solution to the original problem (cf. \cite{Te}). For further details on the application of the Galerkin method to parabolic and elliptic problems in the framework of variational inequalities, we refer the reader to \cite{TRo}. The authors in \cite[Section 7.3, Chapter 7]{DGDM} examined the existence of weak solutions for domain-type hemivariational inequalities associated with parabolic systems in bounded domains. Furthermore, Mig\'orski-Ochal in \cite{SMAO} obtained an existence result for weak solutions of systems governed by a parabolic hemivariational inequality of boundary type via Galerkin method combined with a regularization procedure (see also \cite{SMAON}). 

In 2016, Mig\'orski et al. \cite{Fang2016} studied a new class of hemivariational inequalities for time-dependent NSE, addressing both domain and boundary HVIs by employing Rothe's method for semi-discretization in time, and surjectivity result for multivalued pseudomonotone operators. Furthermore, using this methodology, Jindal et al. \cite{JSMT} analyzed the well-posedness of boundary hemivariational inequalities for non-stationary two and three-dimensional CBF equations. Recently, the authors in \cite{MHHQLM} investigated the well-posedness of the Stokes hemivariational inequality for incompressible fluid flows with damping in three dimensions, under the restriction $1\leq r\leq 5$, by reformulating the problem in a minimization framework. Moreover, using similar ideas to \cite{MHHQLM}, the author of \cite{MTM} established the well-posedness of stationary two- and three-dimensional CBFeD hemivariational inequalities. Also, the authors in \cite{WWXC} investigated the well-posedness of the hemivariational inequality associated with the stationary NSE including a nonlinear damping term via finite element method. Further, the authors in \cite{WAMTM} studied the hemivariational inequality of the stationary CBFeD model via mixed finite element method and also discussed optimal control of stationary CBFeD with hemivariational inequalities in \cite{WAMTM1}.

Despite these contributions, the theoretical analysis of CBF equations coupled with domain hemivariational inequalities for viscous incompressible flows remains largely undeveloped. To address this gap, the present work introduces and investigates the existence and uniqueness of weak solutions for NSHVIs with damping and pumping effects, thereby providing new insights into the mathematical theory of such fluid flow models. To the best of our knowledge, domain-type hemivariational inequalities in this context have not been investigated so far. Motivated by this observation, our work extends the existing theory by employing the Galerkin method to analyze time-dependent hemivariational inequalities for the CBFeD equations \eqref{eqn-dom}, which are of domain type.
\subsection{Difficulties, approaches and novelties}
\begin{itemize}
\item The authors in \cite{Fang2016} established the existence of a weak solution by using Rothe's method for the 2D NSE with domain-type hemivariational inequalities under the additional condition $\mu > \sqrt{2}c_0 \max\{\sqrt{|\mathfrak{D}|},1\}\|l\|^2,$ where $ |\mathfrak{D}|$ is the
measure of $\mathfrak{D}$, $l:\V\to {\L}^2(\Gamma)$ is the trace operator and $c_0 >0$. In contrast, the approach adopted in the present work avoids imposing such a restriction on the viscosity parameter $\mu$. We establish the existence of a weak solution by employing analytical techniques similar to those developed in \cite{DGDM,SMAO} without getting any restriction on $\mu$. The main difficulty in our analysis lies in handling the term
 \begin{align*}
 	 \int_0^t\sum_{i=1}^{d}\int_{\mathfrak{D}} \theta_{in}(s, \u_{n,i}(s)) \u_{n,i}(s)\d x \d s 
 	  \end{align*}
appearing in the inequality \eqref{415}. In the earlier works such as \cite{Fang2016}, the a-priori estimates are obtained by using the bound \eqref{reg-est-3}, which leads to the additional restriction on the viscosity coefficient $\mu$. In the present work, however, we overcome this difficulty by applying Lemma \ref{lemma-1} to derive the necessary a-priori estimates, as demonstrated in \textbf{Step II} of Theorem \ref{thm 4.3}. 

% Then, by using the Banach-Alaoglu theorem to extract weak convergent subsequence and some compactness results, one can extract a further subsequence which converges strongly, so that the weak limit serves as a solution to the original problem .

\item In \cite{SNAHB}, the authors studied the NSE modified by the absorption term $|\u|^{r-2}\u$, for $r > 2$ in bounded domains with compact boundary. They proved the existence of Leray-Hopf weak solutions for dimension $d \geq 2$ and its uniqueness for $d = 2$. However, in three dimension, they were unable to establish the energy equality satisfied by the weak solutions.

Recently, the authors 
%in \cite{FHR} constructed functions that can approximate functions defined on smooth bounded domains by elements of eigenspaces of linear operators (such as the Laplacian or Stokes operator) in such
%a way that the approximations are bounded and converge in both Sobolev and Lebesgue spaces
%simultaneously. As an application, they showed that every weak solution of the critical CBFe ($r = 3$) on a bounded domain in $\R^3$ satisfies the energy equality. Furthermore, 
in \cite{SGMTM} established the existence and uniqueness of global weak solutions (without non-linear damping) satisfying the energy equality to the system \eqref{eqn-dom} for all  $\beta,\mu> 0$, when 
$ r > 3$ and for $2\beta \mu \geq 1 $, when $r = 3$. For more information, one can see \cite{FHR}. Thus, the novelty of our work is to prove the energy equality by approximating the solution via finite-dimensional spaces spanned by the first $n$ eigen functions (see Proposition \ref{prop-energy}) for $r\geq1$ in 2D and $r\geq 3$ in 3D without any additional assumption $2\beta \mu \geq 1 $, when $r = 3$.

\item  Leray and Hopf \cite{JL} established the existence of at least one weak solution for the three-dimensional NSE. However, the global existence and uniqueness of classical solutions to the 3D NSE (see \cite{LSM}) remains one of the most challenging open problems for the mathematical community. Consequently, several modified versions of the 3D NSE have been proposed and studied (cf. \cite{SNAHB, ZZXW}). In \cite{MTT}, the authors considered the classical 3D NSE with nonlinear damping $\beta|\u|^{r-1}\u$ and pumping $\alpha|\u|^{q-1}\u$ (without linear damping) and established the existence of weak solutions for $r>q$ and uniqueness for $r>3$. 

%Furthermore, the authors in \cite{SMAON} investigated the NSE with boundary-type hemivariational inequalities using the Galerkin method, but the uniqueness of solutions was not addressed. 
In \cite{Fang2016}, the authors established the uniqueness of weak solutions under the additional restriction $\mu > m\|l\|^2$, where $m$ is the constant appearing in the monotonicity condition (see \cite[Lemma 6.2]{Fang2016}). The novelty of the present work lies in establishing the uniqueness of weak solutions to problem \eqref{eqn-dom} together with condition \eqref{theta-Cond}, for $r \ge 1$ in the two-dimensional case and for $r \ge 3$ in the three-dimensional case. In particular, when $r=3$, the uniqueness result is obtained under the additional assumption $2\beta\mu > 1$. Unlike \cite{Fang2016}, our analysis does not impose any restriction on the viscosity coefficient $\mu$. To the best of our knowledge, such a result has not been addressed in the existing literature.
\end{itemize}
\subsection{Organization of the paper}
The rest of the paper is organized as follows. In the next section, we present the essential mathematical preliminaries and set up the functional framework. Since the framework used here for the CBF equations differs from that employed for NSE (as detailed in \cite{Te}), we provide an in-depth discussion of the function spaces involved. We then proceed with a thorough analysis of the linear, bilinear, and nonlinear operators that appear in our formulation and dedicated to exploring key properties of these operators.

Section \ref{sec3} focuses on proving the existence and uniqueness of a weak solution to the domain hemivariational inequality associated with 2D and 3D CBFeD equations. Using  Galerkin discretization scheme, we first establish the existence of a solution to an abstract hemivariational inequality, as shown in Theorem \ref{thm 4.3}. Then, we prove the energy equality stated in Proposition \ref{prop-energy} for $r\geq1$ in 2D and $r\geq 3$ in 3D and the corresponding uniqueness results are presented in Theorem \ref{lemma4.6}.

\section{Mathematical Formulation}\setcounter{equation}{0}
Let us first introduce the functional framework needed to establish global solvability of system \eqref{eqn-dom}, and then define the associated operators with their fundamental properties.

\subsection{Function spaces} Let $\mathrm{C}_0^{\infty}(\mathfrak{D};\R^d)$ be the space of all infinitely differentiable $\R^d$-valued functions having compact support contained in $\mathfrak{D}\subset\R^d$.  Let us define 
\begin{align*} 
	\mathcal{V}&:=\{\u\in \mathrm{C}_0^{\infty}(\mathfrak{D},\R^d):\nabla\cdot\u=0\},\\
	\H&:= \overline{\mathcal{V}}^{\|\cdot\|_{\L^2(\mathfrak{D})}}, \quad \V:= \overline{\mathcal{V}}^{\|\cdot\|_{\mathbb{H}_0^1(\mathfrak{D})}} \text{ and } \widetilde{\L}^{p}:= \overline{\mathcal{V}}^{\|\cdot\|_{\L^p(\mathfrak{D})}} \text{ for } p\in(2,\infty),
\end{align*} 
where $\overline{\mathcal{V}}$ denotes the closure of $ \mathcal{V} $ in their respective norms.
The spaces $\H$, $\V$ and $\widetilde{\L}^p$ can be  characterize as 
\begin{align*}
	\H&:=\{\u\in\L^2(\mathfrak{D}):\nabla\cdot\u=0,\u\cdot\n\big|_{\partial\mathfrak{D}}=0\},  \\
	\V&:=\{\u\in\mathbb{H}_0^1(\mathfrak{D}):\nabla\cdot\u=0\}, \\
	\widetilde{\L}^p&:=\{\u\in\L^p(\mathfrak{D}):\nabla\cdot\u=0, \u\cdot\n\big|_{\partial\mathfrak{D}}=0\}, 
\end{align*}
where $\n$ is the outward normal to $\partial\mathfrak{D}$ and $\u\cdot\n\big|_{\partial\mathfrak{D}}$ is understood in the sense of trace (\cite[Lemma 1.3, Chapter 1]{Te}). We denote the norms in $\H, \V  \text{ and } \widetilde{\L}^p$ by
$$\|\u\|_{\H}^2:=\int_{\mathfrak{D}}|\u(x)|^2\d x, \ \|\u\|_{\V}^2:=\int_{\mathfrak{D}}|\nabla\u(x)|^2\d x\  \text{ and } \
\|\u\|_{\widetilde{\L}^p}^p=\int_{\mathfrak{D}}|\u(x)|^p\d x ,$$ respectively.
Let $(\cdot,\cdot)$ denote the inner product in the Hilbert space $\H$ and $\langle \cdot,\cdot\rangle $ represent  the induced duality between the spaces $\V^{\prime}$  and its dual $\V$ as well as $\widetilde{\L}^{p^{\prime}}$ and its dual $\widetilde{\L}^{p}$, where $\frac{1}{p}+\frac{1}{p^{\prime}}=1$. 
%From \cite[Subsection 2.1]{FKS}, the sum space $\V^{\prime}+\widetilde{\L}^{p^{\prime}}$ is well-defined and  is a Banach space with respect to the norm 
%\begin{align*}
%	\|\u\|_{\V^{\prime}+\widetilde{\L}^{p^{\prime}}}&:=\inf\{\|\u_1\|_{\V^{\prime}}+\|\u_2\|_{\wi\L^{p^{\prime}}}:\u=\u_1+\u_2, \u_1\in\V^{\prime} \ \text{and} \ \u_2\in\wi\L^{p^{\prime}}\}\nonumber\\&=
%	\sup\left\{\frac{|\langle\u_1+\u_2,\boldsymbol{f}\rangle|}{\|\boldsymbol{f}\|_{\V\cap\widetilde{\L}^p}}:\boldsymbol{0}\neq\boldsymbol{f}\in\V\cap\widetilde{\L}^p\right\},
%\end{align*}
%where $\|\cdot\|_{\V\cap\widetilde{\L}^p}:=\max\{\|\cdot\|_{\V}, \|\cdot\|_{\wi\L^p}\}$ is a norm on the Banach space $\V\cap\widetilde{\L}^p$. Also the norm $\max\{\|\u\|_{\V}, \|\u\|_{\wi\L^p}\}$ is equivalent to the norms  $\|\u\|_{\V}+\|\u\|_{\widetilde{\L}^{p}}$ and $\sqrt{\|\u\|_{\V}^2+\|\u\|_{\widetilde{\L}^{p}}^2}$ on the space $\V\cap\widetilde{\L}^p$. 
Moreover, we have the following continuous embeddings $$\V\cap\widetilde{\L}^p\hookrightarrow\V\hookrightarrow\H\cong\H'\hookrightarrow\V^{\prime}\hookrightarrow\V^{\prime}+\widetilde{\L}^{p^{\prime}},$$ where the embedding $\V\hookrightarrow\H$ is compact. 

\subsection{Linear operator}
Let us define a bilinear form $\mathtt{a}(\cdot,\cdot): \V \times \V\to\R$ by $$\mathtt{a}(\p, \q) := (\nabla \p, \nabla \q),\  \text{ for }\ \p, \q \in\V.$$
% From the definition of $\mathtt{a}(\cdot,\cdot)$, it satisfies
%\begin{align}\label{eqn-abound}
%	|\mathtt{a}(\p, \q)|\leq \|\p\|_{\V}\|\q\|_{\V},\ \text{ for all }\ \p, \q \in\V.
%\end{align} 
In view of \emph{the Riesz representation theorem}, there exists a unique linear operator $\A:\V\to \V^{\prime}$ such that
\begin{align*}
	\mathtt{a}(\p, \q) = \langle\A\p, \q\rangle, \ \text{ for all }\ \p, \q \in\V.
\end{align*}
%Moreover, it satisfies
%\begin{align}\label{eqn-coercive}
%	\mathtt{a}(\p, \p)=\|\nabla \p\|_{\H}^2=\|\p\|_{\V}^2\  \text{ for all }\ \p \in\V.
%\end{align}
%Therefore, by means of the \emph{Lax-Milgram theorem} (see \eqref{eqn-abound} and \eqref{eqn-coercive}), the operator $\A : \V\to\V^{\prime}$ is an isomorphism.
%   
\subsection{Bilinear operator}
Let us define the \emph{trilinear form} $\mathtt{b}(\cdot,\cdot,\cdot):\V\times\V\times\V\to\R$ by $$\mathtt{b}(\p,\q,\r)=\int_{\mathfrak{D}}(\p(x)\cdot\nabla)\q(x)\cdot\r(x)\d x=\sum_{i,j=1}^d\int_{\mathfrak{D}}\p_i(x)\frac{\partial \q_j(x)}{\partial x_i}\r_j(x)\d x.$$ 
If $\p, \q$ are such that the linear map $\mathtt{b}(\p, \q, \cdot) $ is continuous on $\V$, the corresponding element of $\V^{\prime}$ is denoted by $\B(\p, \q)$. We represent $\B(\p) = \B(\p, \p)=(\p \cdot\nabla)\p$.
An integration by parts gives 
\begin{equation*}
	\left\{
	\begin{aligned}
		\mathtt{b}(\p,\q,\q) &= 0 \ \text{ for all }\ \p,\q \in\V,\\
		\mathtt{b}(\p,\q,\r) &=  -\mathtt{b}(\p,\r,\q)\ \text{ for all }\ \p,\q,\r\in \V.
	\end{aligned}
	\right.
\end{equation*}
%In the trilinear form, an application of H\"older's inequality yields
%\begin{align*}
%		|\mathtt{b}(\u,\v,\w)|=|\mathtt{b}(\u,\w,\v)|\leq \|\u\|_{\widetilde{\L}^{r+1}}\|\v\|_{\widetilde{\L}^{\frac{2(r+1)}{r-1}}}\|\w\|_{\V},
%\end{align*}
%for all $\u\in\V\cap\widetilde{\L}^{r+1}$, $\v\in\V\cap\widetilde{\L}^{\frac{2(r+1)}{r-1}}$ and $\w\in\V$.
%so that we get 
%\begin{align}\label{2p9}
%	\|\B(\u,\v)\|_{\V'}\leq \|\u\|_{\widetilde{\L}^{r+1}}\|\v\|_{\widetilde{\L}^{\frac{2(r+1)}{r-1}}}.
%\end{align}
\begin{lemma}[\cite{MTMS}]
The trilinear map $\mathtt{b} : \V\times\V\times\V \to \R$ has a unique extension from $(\V\cap\widetilde{\L}^{r+1})\times(\V\cap\widetilde{\L}^{\frac{2(r+1)}{r-1}})\times\V$ to $\R$ which is a bounded trilinear map and $\B :\V\cap\widetilde{\L}^{r+1} \to \V^{\prime}+\widetilde{\L}^{\frac{r+1}{r}}$ satisfies the following properties:
\end{lemma} 
%By using H\"older's and interpolation inequalities, 
\begin{itemize} 
\item[(i)] We estimate $\left|\langle \B(\p,\p),\q \rangle \right|$ as
\begin{align}\label{bound-B}
\left|\langle \B(\p,\p),\q\rangle \right|=\left|\mathtt{b}(\p,\p,\q)\right|\leq \|\p\|_{\widetilde{\L}^{r+1}}\|\p\|_{\widetilde{\L}^{\frac{2(r+1)}{r-1}}}\|\q\|_{\V}\leq\|\p\|_{\widetilde{\L}^{r+1}}^{\frac{r+1}{r-1}}\|\p\|_{\H}^{\frac{r-3}{r-1}}\|\q\|_{\V},
\end{align}
for all $\q\in\V\cap\widetilde{\L}^{r+1}$.
\item[(ii)] For $r \geq 3$, the operator $\B :\V\cap\widetilde{\L}^{r+1} \to \V^{\prime}+\widetilde{\L}^{\frac{r+1}{r}}$ is a locally Lipschitz, that is, 
\begin{align}\label{lip}
	\|\B(\p)-\B(\q)\|_{\V^{\prime}+\widetilde{\L}^{\frac{r+1}{r}}}&\leq \left(\|\p\|_{\H}^{\frac{r-3}{r-1}}\|\p\|_{\widetilde{\L}^{r+1}}^{\frac{2}{r-1}}+\|\q\|_{\H}^{\frac{r-3}{r-1}}\|\q\|_{\widetilde{\L}^{r+1}}^{\frac{2}{r-1}}\right)\|\p-\q\|_{\widetilde{\L}^{r+1}}.
\end{align}
\end{itemize}

\subsection{Nonlinear operator}
Let us now define an operator $$\C(\p):=|\p|^{r-1}\p\ \text{ for }\ \p\in\wi\L^{r+1}.$$ 
%For convenience of notation, we use $\C$ for $\C_r$ in the rest of the paper.  
It follows immediately that
\begin{align*}
\langle\C(\p),\p\rangle =\|\p\|_{\widetilde{\L}^{r+1}}^{r+1}.
\end{align*}
Moreover, we define $\wi{\C}(\p)=|\p|^{q-1}\p$ for $1\leq q<r<\infty$. The nonlinear operator $\wi{\C}(\cdot)$ satisfies  similar properties as $\C(\cdot)$.

\begin{lemma}[\cite{MTMS}] The nonlinear operator $\C(\cdot):\widetilde{\L}^{r+1}\to\widetilde{\L}^{\frac{r+1}{r}}$ is locally Lipschitz, that is, 
	%	By making the use of \eqref{MVT} and H\"older's inequality, we finally calculate
	\begin{align}\label{213}
		\langle |\p|^{r-1}\p-|\q|^{r-1}\q,\p-\q\rangle
		&\leq r\left(\|\p\|_{\widetilde{\L}^{r+1}}+\|\q\|_{\widetilde{\L}^{r+1}}\right)^{r-1}\|\p-\q\|_{\widetilde{\L}^{r+1}}^2,
	\end{align}
	for all $\p,\q\in\widetilde{\L}^{r+1}$.
	%Thus the operator $\C(\cdot):\widetilde{\L}^{r+1}\to\widetilde{\L}^{\frac{r+1}{r}}$ is locally Lipschitz.
\end{lemma}

\begin{lemma}[{\cite[Section 2.4]{MTMS}}]\label{Cmono1}
	For all $\p,\q\in \wi\L^{r+1}$ and $r\geq 1$, we have 
	\begin{align*}
		&\langle\p|\p|^{r-1}-\q|\q|^{r-1},\p-\q\rangle\geq \frac{1}{2}\||\p|^{\frac{r-1}{2}}(\p-\q)\|_{\H}^2
		+\frac{1}{2}\||\q|^{\frac{r-1}{2}}(\p-\q)\|_{\H}^2\geq 0.
	\end{align*}
\end{lemma}
 
\subsection{Clarke subdifferential}
In this subsection, we introduce the Clarke subdifferential and provide its characterization in terms of generalized directional derivative, as outlined below:

\begin{definition}[{\cite[Definition 3.20]{SMAOMS}}]
Let $\mathcal{E}$ be a normed space. A function $\mathcal{G}:\mathcal{E}\to\mathbb{R}$ is called locally Lipschitz, if for every $\mathpzc{x}\in\mathcal{E}$, there exists a neighborhood $N$ of $\mathpzc{x}$ and a constant $L_N$ such that $$|\mathcal{G}(\mathpzc{y})-\mathcal{G}(\mathpzc{z})|\leq L_N\|\mathpzc{y}-\mathpzc{z}\|_{\mathcal{E}}\ \text{ for all }\ \mathpzc{y}, \mathpzc{z} \in N.$$ 
\end{definition}

In analogous with the Clarke generalized directional derivative and generalized gradient for locally Lipschitz functions, we recall the following definitions:
\begin{definition}[{\cite{SMAOMS}}]\label{cLarkedef}
	Let $\mathcal{G}(\cdot, \cdot) :(0,T) \times\mathcal{E} \to\R$ be a function such that $\mathcal{G}(\cdot, \mathpzc{x})$ is measurable on $(0,T) $ for all $\mathpzc{x} \in \mathcal{E}$ and $\mathcal{G}(t, \cdot)$ is locally Lipschitz  for all $t \in (0,T)$. The generalized directional derivative of $\mathcal{G}$ at a point $\mathpzc{x}\in\mathcal{E}$ in the direction $\mathpzc{z}\in\mathcal{E}$, denoted by $\mathcal{G}^0(t, \mathpzc{x};\mathpzc{z})$, is defined by
	\begin{align*}
		\mathcal{G}^0(t,\mathpzc{x};\mathpzc{z})=\lim_{\mathpzc{y}\to\mathpzc{x}}\sup_{\lambda\downarrow 0}\frac{\mathcal{G}(t,\mathpzc{y}+\lambda\mathpzc{z})-\mathcal{G}(t,\mathpzc{y})}{\lambda}. 
	\end{align*}
	The generalized gradient or Clarke subdifferential of $\mathcal{G}$ at $\mathpzc{x}$, denoted by $\partial_{\mathpzc{C}}\mathcal{G}(t, \mathpzc{x})$ for all $(0,T) \times\mathcal{E} $, is the subset of the dual space $\mathcal{E}^{\prime}$ given by
	\begin{align}\label{def-gra}
		\partial_{\mathpzc{C}} \mathcal{G}(t,\mathpzc{x})=\left\{\boldsymbol{\zeta}\in\mathcal{E}^{\prime}:\mathcal{G}^0(t,\mathpzc{x};\mathpzc{z})\geq {}_{\mathcal{E}^{\prime}}\langle\boldsymbol{\zeta},\mathpzc{z}\rangle_{\mathcal{E}}\ \text{ for all }\ \mathpzc{z}\in\mathcal{E} \right\}. 
	\end{align}
	A locally Lipschitz function $\mathcal{G}$ is said to be \emph{regular (in the sense of Clarke)} at a point $\mathpzc{x} \in \mathcal{E}$,  if for every direction $\mathpzc{z}\in\mathcal{E}$, the one-sided directional derivative $\mathcal{G}^{\prime}(t,\mathpzc{x};\mathpzc{z})$  exists and satisfies $$\mathcal{G}^0(t,\mathpzc{x};\mathpzc{z}) = \mathcal{G}^{\prime}(t,\mathpzc{x};\mathpzc{z}).$$
\end{definition}

\begin{proposition}[{\cite[Proposition 3.44]{SMAOMS}}]
	If  $\mathcal{G}(\cdot, \cdot) :(0,T) \times\mathcal{E} \to\R$ is a function such that $\mathcal{G}(\cdot, \mathpzc{x})$ is measurable on $(0,T) $ for all $x \in \mathcal{E}$ and $\mathcal{G}(t, \cdot)$ is locally Lipschitz for all $t \in (0,T)$, then 
	\begin{enumerate}
		\item  for every $(t,\mathpzc{x}) \in (0,T) \times\mathcal{E}, ~\partial_C \mathcal{G}(t, \mathpzc{x})$ is a nonempty, convex, weak$^*$-compact subset of $\mathcal{E}^{\prime}$ and satisfies $$\|\boldsymbol{\zeta}\|_{\mathcal{E}^{\prime}}\leq L_N \ \text{ for all }\  \boldsymbol{\zeta}\in\partial_C \mathcal{G}(t, \mathpzc{x}); \ \text{ and }$$
		\item we have 
		$$	\mathcal{G}^0(t,\mathpzc{x};\mathpzc{z})=\max\left\{\langle\boldsymbol{\zeta},\mathpzc{z}\rangle:\boldsymbol{\zeta}\in\partial_C \mathcal{G}(t,\mathpzc{x})\right\} \ \text{ for all }\  \mathpzc{z}\in\mathcal{E}.$$
	\end{enumerate}
\end{proposition}

\subsection{Preliminaries} Let $\E$ be a Banach space and $T>0$. Let us $\mathcal{D}'(0,T;\E)$ denote the space of all distributions from $(0,T)$ to a Banach space $\E$.
%the Bochner space $\mathrm{L}^p(0,T;\E),$ $1\leq p<\infty$ is defined as: 
%\begin{align*}
%	\mathrm{L}^p(0,T;\E):=\left\{\u:(0,T)\to\E: \u\ \text{ is strongly measurable and }\ \int_0^T\|\u(t)\|_{\E}^p\d t<\infty \right\}.
%\end{align*}
%For $p=\infty$, we define the set 
%\begin{align*}
%	\mathrm{L}^{\infty}(0,T;\E):=\left\{\u:(0,T)\to\E: \u\ \text{ is strongly measurable and }\ \esssup_{t\in(0,T)}\|\u(t)\|_{\E}<\infty \right\}.
%\end{align*}
The following result is a generalization of the well-known Lions-Magenes Lemma \cite{JLEM}.  
\begin{theorem}[{\cite[Theorem 1.8]{VVMI}}]\label{Thm-Abs-cont}
	Let $\H$ be a Hilbert space identified with its dual $\H^\prime$, and $\V, \X, \E$ be Banach spaces and satisfies the  following inclusions
	\[\V \hookrightarrow \H \hookrightarrow \V^\prime \hookrightarrow \E\  \mbox{ and } \ \X \hookrightarrow \H \hookrightarrow \X^\prime \hookrightarrow \E,\]
	where the spaces $\V^\prime$ and $\X^\prime$ are the duals of $\V$ and $\X$, respectively. Assume that $p>1$ and $\u\in \mathrm{L}^{2}(0,T;\V)\cap \mathrm{L}^p (0,T; \X)$,  $\u^{\prime}\in \mathcal{D}^\prime(0,T; \E)$ and $\u^\prime = \u_1 +\u_2$, where $\u_1 \in \mathrm{L}^{2}(0,T;\V^\prime)$ and $\u_2 \in \mathrm{L}^{p^\prime}(0,T;\X^{\prime})$. Then, $\u \in \mathrm{C}([0,T]; \H).$
\end{theorem}

\begin{theorem}[{\cite[Proposition 2]{PKa}, \cite[Theorem 5]{JSi}}]\label{thm-Simon}
	Let $1\leq p,q <\infty$  and $\E_1 \hookrightarrow \E_2 \hookrightarrow \E_3$ be real Banach spaces such that $\E_1$ is reflexive, $\E_1 \hookrightarrow \E_2$  is compact embedding and $\E_2 \hookrightarrow \E_3$ is continuous. Then the following condition hold: 
	\begin{enumerate}
		\item  If $\mathcal{H}$ is a bounded subset of $\mathscr{W}^{p,q}(0,T;\E_1,\E_3):=\{\u\in\mathrm{L}^p(0,T;\E_1): \u^{\prime}\in\mathrm{L}^q(0,T;\E_3)\},$  then it is relatively compact in $\mathrm{L}^p(0,T;\E_2)$.
	\end{enumerate}
\end{theorem}

%The following  result will be used to establish the covergence of the nonlinear term. 
\begin{lemma}[{\cite[Lemma 1.3]{JLL}}]\label{Lem-Lions}
Consider $\mathfrak{D}_T \subset \R^d\times \R$ is a bounded open set, and let $\varphi_m$, $\varphi$ be functions in $\mathrm{L}^q(\mathfrak{D}_T)$ with $1<q <\infty$ for $m\in\N$, such that
	\begin{align*}
		\|\varphi_m\|_{\mathrm{L}^q(\mathfrak{D}_T)} \leq C,\ \ \mbox{for all}\ \ m\in\N \; \; \mbox{and} \ \ \varphi_m \to \varphi\;\; \mbox{a.e. in}\ \ \mathfrak{D}_T,\  \mbox{as}\ \ m \to \infty.
	\end{align*}
Then, $\varphi_m \xrightarrow{w} \varphi$ in $\mathrm{L}^q(\mathfrak{D}_T)$, as $m\to \infty$.
\end{lemma}

\section{A domain hemivariational inequality}\label{sec3}\setcounter{equation}{0}
In this section, we examine the domain hemivariational inequality associated with problem \eqref{eqn-dom}. Let us first provide the variational (or weak formulation) of \eqref{eqn-dom}. By assuming sufficient regularity of the functions involved, we multiply \eqref{eqn-dom} by $\v\in\V\cap\wi\L^{r+1}$ to deduce
\begin{align}\label{varform}
\frac{\d}{\d t}\langle\u,\v\rangle +\langle\mu\A\u+\B(\u)+\alpha\wi{\C}(\u)+\beta\C(\u),\v\rangle=
\langle\f(t)+\g(t),\v\rangle.
\end{align}
For the ease of notation, we set $ j(x,t,\u(x,t)) \equiv j(t,\u(t))$. 
From the definition of generalized gradient \eqref{def-gra}, we write
\begin{align}\label{cLaerdef1}
	\partial_{\mathpzc{C}} j(\cdot,\u)=\left\{\boldsymbol{\zeta}\in\R^d:
	j^0(\cdot,\u;\v)\geq \boldsymbol{\zeta}\cdot\v\ \text{ for all }\ \v\in\R^d \right\}. 
\end{align}
In view of \eqref{eqn-sub-diff}, since $-\g(\cdot)\in\partial_{\mathpzc{C}} j(\cdot,\u)$,   for $\u=(\u_1,\ldots,\u_d)$, from the \eqref{cLaerdef1}, we have 
\begin{align}\label{eqn-clarke}
	-\int_{\mathfrak{D}}\g(x,t)\cdot\v(x)\d x\leq 
	\int_{\mathfrak{D}}j^0(t, \u(t);\v)\d x= \sum_{i=1}^{d}
	\int_{\mathfrak{D}}j_i^0(t,\u_i(t);\v_i)\d x, 
\end{align}
where $j^0(t,\boldsymbol{\xi};\boldsymbol{\zeta})\equiv j^0(x,t,\boldsymbol{\xi};\boldsymbol{\zeta})$ denotes the generalized directional derivative of $j(x,t,\cdot)$ at the point $\boldsymbol{\xi}\in\R^d$ in the direction $\boldsymbol{\zeta}\in\R^d$. Note that in order to make sense of the left most integral term in \eqref{eqn-clarke}, $\g$ must be $\H-$valued. 
Let us define a set 
\begin{align*}
	\mathcal{W}:=&\left\{\u\in\mathrm{L}^{\infty}(0,T;\H)\cap\mathrm{L}^2(0,T;\V)\cap\mathrm{L}^{r+1}(0,T;\wi\L^{r+1})\big|\right.\nonumber\\&\qquad\left.
	\frac{\d\u}{\d t}\in\mathrm{L}^2(0,T;\V^{\prime})+ \mathrm{L}^{\frac{r+1}{r}}(0,T;\wi\L^{\frac{r+1}{r}})\right\}.
\end{align*} 
From Theorem \ref{Thm-Abs-cont}, we infer that the embedding $\mathcal{W}  \hookrightarrow \mathrm{C}([0,T];\H)$ is continuous. Consider $\u^0\in\H$ and $\f\in\mathrm{L}^2(0,T;\V^{\prime})$ are given. From \eqref{varform} and \eqref{eqn-clarke},
we deduce the following variational formulation of \eqref{eqn-dom}.
\begin{problem}\label{prob-inequality}
	Find $\u\in\mathcal{W}$ such that 
	\begin{equation*}
		\left\{
		\begin{aligned}
		&	\frac{\d}{\d t}\langle\u,\v\rangle +\langle\mu\A\u
			+\B(\u)+\alpha\wi{\C}(\u)+\beta\C(\u),\v\rangle
			\\&\quad+
			\sum_{i=1}^{d}
			\int_{\mathfrak{D}}j_i^0(t,\u_i(t);\v_i)\d x \geq\langle\boldsymbol{f}(t),\v\rangle, \ \text{ for a.e. } \ t\in[0,T],\\
		&	\u(0)=\u^0,
		\end{aligned}
		\right.
	\end{equation*}
	for all $\v\in\V\cap\wi\L^{r+1}$. 
\end{problem}
{We now provide the following definition of weak solution, inspired by \cite[Defintion 2.1]{GGP} and \cite[Definition 1.1.1, Chapter V]{HS}.
\begin{definition}\label{weakdef}
Let $1 \leq q < r $. A function $\u\in\mathcal{W}$ is called a \emph{weak solution} to the system \eqref{eqn-dom}, if for 
a given $\f\in\mathrm{L}^2(0,T;\V^{\prime})$ and $\u^0\in\H$, there exists $-\g(\cdot)\in\partial_{\mathpzc{C}} j(\cdot,\u)$ with values in $\H$ such that  $\u(\cdot)$ satisfies the following weak formulation:
	\begin{align}\label{3.6}
		-&\int_{0}^{t} \int_{\mathfrak{D}} \u(x,s) \partial_t \boldsymbol{\phi}(x,s)\d x\d s
		+ \mu \int_{0}^{t}\int_{\mathfrak{D}} \nabla \u(x,s)\cdot \nabla \boldsymbol{\phi}(x,s)\d x\d s \nonumber \\
		& +\int_{0}^{t}\int_{\mathfrak{D}} \big((\u(x,s)\cdot\nabla)\u(x,s)\big)\cdot \boldsymbol{\phi}(x,s)\d x\d s 	-\int_{0}^{t}(\g(s),\boldsymbol{\phi}(s))\d s \nonumber\\
		& +\int_{0}^{t}\int_{\mathfrak{D}}
		\left(\alpha |\u(x,s)|^{q-1}\u(x,s)
		+\beta |\u(x,s)|^{r-1}\u(x,s)\right)\cdot\boldsymbol{\phi}(x,s)\d x\d s \nonumber\\
		=& \int_{\mathfrak{D}} \u(x,0)\boldsymbol{\phi}(x,0)\d x
		-\int_{\mathfrak{D}} \u(x,t)\boldsymbol{\phi}(x,t)\d x
		+\int_{0}^{t}\langle \f(s),\boldsymbol{\phi}(s)\rangle\d s ,
	\end{align}
for all $t\in[0,T]$ and test functions  $\boldsymbol{\phi} \in \mathrm{C}_{0}^{\infty}([0,T);\mathcal{V})$. Moreover, the initial data is satisfied in the following sense: 
\begin{align*}
	\lim\limits_{t\to0}(\u(t),\v)=(\u^0,\v)\ \text{ for all } \ \v\in\H.
\end{align*}
\end{definition}

\begin{remark}
1.) The regularity 
$$\u \in \mathrm{L}^{\infty}(0,T;\H)\cap \mathrm{L}^{2}(0,T;\V) \cap \mathrm{L}^{r+1}(0,T;\wi\L^{r+1})$$
and
	$$\frac{\d\u}{\d t} \in\mathrm{L}^{2}(0,T;\V^{\prime})+\mathrm{L}^{\frac{r+1}{r}}(0,T;\widetilde\L^{\frac{r+1}{r}})\hookrightarrow\mathrm{L}^{\frac{r+1}{r}}(0,T;
		\V^{\prime}+\wi\L^{\frac{r+1}{r}})
	$$
	implies that 
$$
	\u \in \mathrm{W}^{1,\frac{r+1}{r}}(0,T;	\V^{\prime}+\wi\L^{\frac{r+1}{r}})
	\hookrightarrow \mathrm{C}([0,T];\V^{\prime}+\wi\L^{\frac{r+1}{r}}),
	$$
where we have used \cite[Theorem 2, pp. 302]{LCE}. Since $\H$ is reflexive and the embedding 
	$\H \hookrightarrow \V^{\prime}+\wi\L^{\frac{r+1}{r}}$
	is continuous, it follows from \cite[Proposition 1.7.1]{PCAM} that $
	\u \in  \mathrm{C}_{w}([0,T];\H).$
%	where $ \mathrm{C}_{w}([0,T];\H)$ denotes the space of functions 
%	$\u:[0,T]\to \H$ which are weakly continuous. 
That is, for all 
	$\boldsymbol{\upsilon} \in \H$, the scalar function 
	$$
	[0,T] \ni t \mapsto (\u(t),\boldsymbol{\upsilon}) \in \mathbb{R}
	$$
	is continuous on $[0,T]$. Thus, the first two terms on the right-hand side of \eqref{3.6} are well defined.
	
2.) Using the density of $\mathcal{V} \subset \V\cap\wi\L^{r+1}$, it follows from \cite[Lemmas 2.1 and 2.2]{GGP} that \eqref{3.6} is equivalent to the following formulation: for all $t \in (0,T]$
	\begin{align}\label{3.7}
		&\mu \int_{0}^{t}(\nabla \u(s), \nabla \boldsymbol{\phi} ) \d s
		+ \int_{0}^{t} \langle (\u(s)\cdot \nabla)\u(s), \boldsymbol{\phi} \rangle \d s-\int_{0}^{t}(\g(s),\boldsymbol{\phi})\d s  \nonumber\\
		&\quad+ \int_{0}^{t} 
		\left\langle 
		\alpha |\u(s)|^{q-1} \u(s) + \beta |\u(s)|^{r-1} \u(s),\boldsymbol{\phi}
		\right\rangle \d s  \nonumber \\
		&= (\u(0),\boldsymbol{\phi}) - (\u(t),\boldsymbol{\phi}) + \int_{0}^{t} \langle \f(s),\boldsymbol{\phi} \rangle\d s ,
	\end{align}
	for all $\boldsymbol{\phi} \in \V\cap\wi\L^{r+1}$.	Since $-\g(\cdot)\in\partial_{\mathpzc{C}} j(\cdot,\u)$, then by applying  \eqref{eqn-clarke}, the formulation \eqref{3.7} further reduces to 
	\begin{align*}
	&\mu \int_{0}^{t}(\nabla \u(s), \nabla \boldsymbol{\phi} ) \d s
	+ \int_{0}^{t} \langle (\u(s)\cdot \nabla)\u(s), \boldsymbol{\boldsymbol{\phi}} \rangle \d s+	\sum_{i=1}^{d}
	\int_{0}^{t} \int_{\mathfrak{D}}j_i^0(s,\u_i(s);\boldsymbol{\boldsymbol{\phi}}_i)\d x\d s  \nonumber\\
	&+ \int_{0}^{t} 
	\left\langle 
	\alpha |\u(s)|^{q-1} \u(s) + \beta |\u(s)|^{r-1} \u(s),\boldsymbol{\phi}
	\right\rangle \d s \geq (\u(0),\boldsymbol{\phi}) - (\u(t),\boldsymbol{\phi}) + \int_{0}^{t} \langle \f(s),\boldsymbol{\phi} \rangle\d s,
\end{align*}
for all $\boldsymbol{\phi} \in \V\cap\wi\L^{r+1}$.
\end{remark}}

We address Problem \ref{prob-inequality} by employing a Galerkin approximation scheme applied to a suitably regularized version of the problem. The main analytical challenge in the variational formulation stems from the presence of the term involving the Clarke subdifferential, which is generally non-smooth and multivalued. To handle this difficulty and to make the Galerkin method applicable, we first introduce a regularization of the Clarke subdifferential. This regularization yields a family of well-posed finite-dimensional problems, enables the derivation of uniform energy estimates, and facilitates the passage to the limit in the approximation process.
 
\subsection{Regularization of Clarke subdifferential}
For $i=1,\ldots,d$, let 
$$\theta_i :\mathfrak{D} \times(0,T) \times \mathbb{R} \to \mathbb{R}, 
\ (x,t,\xi)\mapsto\theta_i(x,t,\xi),$$
be a function which is possibly non-monotone and non-smooth with respect to the variable $\xi$. The following hypothesis on $\theta_i$ is adopted from the framework developed in \cite{SMAO} and \cite[Section 7.3, Chapter 7, pp. 108]{DGDM}, with minor modifications:

\begin{hypothesis}\label{hyptheta} We assume that the function $\theta_i$ satisfies the following hypotheses:
	
	(i) \textbf{Local boundedness:} For every $r>0,$ there exists a constant $c=c(r)>0$ such that 
	\begin{align*}
	|\theta_i(x,t,\xi)| \leq c(r),\ \text{ for all } \ (x,t) \in \mathfrak{D} \times(0,T) \text{ and a.e. } |\xi| \leq r;
	\end{align*}
% (ii) $\theta_i$ is uniformly continuous with respect to $\xi$, that is, there exists  $\varepsilon_0 > 0$ such that for all $(x,t,\xi) \in  \mathfrak{D} \times(0,T) \times \mathbb{R}, \text{ and for all } \delta >0, \text{ there exists } \gamma = \gamma(\delta,x,t,\xi) > 0$ such that
%	$$|\theta_i(x,t,\xi)-\theta_i(x,t,\xi')| <\delta\text{ as } |\xi-\xi'| <\varepsilon_0;$$
 
   (ii) \textbf{Continuity in $(x,t)$:} For a.e. $\xi\in\mathbb{R}$, the mapping $$(x,t) \mapsto \theta_i(x,t,\xi)$$ is continuous on $\mathfrak{D} \times(0,T);$
   
   (iii) \textbf{Growth condition:} $\theta_i\in \mathrm{L}^{\infty}_{\mathrm{loc}}(\mathbb{R})$ is such that the left and right limits $\theta_i(x,t,\xi\pm0)$ exist for every $ \xi \in \mathbb{R}$ and it verifies the growth condition
   \begin{align}\label{growth}
   |\theta_i(x,t,\xi)|\leq\bar{\alpha}(x,t)+C_{1,i} + C_{2,i}|\xi|,\  \text{ for all } \ (x,t,\xi) \in \mathfrak{D} \times(0,T) \times \mathbb{R},
   \end{align}
   with a non-negative function $ \bar{\alpha} \in \mathrm{L}^2(\mathfrak{D} \times(0,T)) \text{ and positive constants } C_{1,i}, C_{2,i};$
   
   (iv) \textbf{Ultimate monotonicity in $\xi$:} For each $i=1,\ldots, d$, the graph $(\xi,\theta_i(x,t,\xi))$ is ultimate increasingly with respect to $\xi$, that is, there exists $\phi_i \geq 0$ such that 
   \begin{align}\label{thetsupinf}
    \esssup\limits_{\xi \in (-\infty,-\phi_i)}\theta_i(x,t,\xi) \leq 0 \leq \essinf\limits_{\xi \in (\phi_i,\infty)}\theta_i(x,t,\xi).
   \end{align}
\end{hypothesis}

Since $\theta_i(x,t,\cdot):\mathbb{R}\to\mathbb{R}$ may be discontinuous with respect to $\xi$, we introduce its regularization by `filling in the jumps'. For $\varepsilon>0$, $(x,t)\in \mathfrak{D}\times(0,T)$ and $\xi\in\mathbb{R},$ we define the lower and upper envelopes
\begin{align*}
	\underline{\theta_i}^{\varepsilon}(x,t,\xi) := \essinf\limits_{|\tau - \xi| \leq \varepsilon} \theta_i(x,t,\tau),\\
	\overline{\theta_i}^{\varepsilon}(x, t,\xi) := \esssup\limits_{|\tau - \xi| \leq \varepsilon} \theta_i(x,t, \tau).
\end{align*} 
For $(x, t)\in\mathfrak{D} \times(0,T) $ and $ \xi \in \mathbb{R}$ fixed, the functions $\underline{\theta_i}^{\varepsilon}$ and $\overline{\theta_i}^{\varepsilon}$ are monotone with respect to $\varepsilon$, being increasing and decreasing, respectively. Consequently, the following limits are well-defined:
\begin{align*}
	\underline{\theta_i}(x, t,\xi):= \lim_{\varepsilon \to 0^{+}} \underline{\theta_i}^{\varepsilon}(x, t,\xi), \\
	\overline{\theta_i}(x,t,\xi ):= \lim_{\varepsilon \to 0^{+}} \overline{\theta_i}^{\varepsilon}(x,t,\xi ).
\end{align*}
For $\varepsilon>0,$ and $t\in\mathbb{R},$ we define the multivalued map $\widehat{\theta_i} \colon \mathfrak{D}\times(0,T)\times\mathbb{R}\multimap \mathbb{R}$ by
\begin{align}\label{Def-theta}
	\widehat{\theta_i}(x,t,\xi)=[\underline{\theta_i}(x,t,\xi), \overline{\theta_i}(x,t,\xi)],
\end{align}
that is, $\widehat{\theta_i}(x,t,\xi)$ is the closed interval with end points given by $\underline{\theta_i}(x,t,\xi)$ and 
$\overline{\theta_i}(x,t,\xi)$, respectively. 

From \cite[Section 2]{KCC}, it has been shown that if the limits $\theta_i(x,t,\xi\pm0)$ exist for every $\xi\in\R$, then one can determine the Clarke subdifferential for a locally Lipschitz function $j_i(x,t,\cdot):\mathbb{R} \to\mathbb{R}$ such that 
\begin{align}\label{Equality}
	\partial_C j_i(x,t,\xi)=\widehat{\theta_i}(x,t,\xi)\  \text{ for } \ \xi \in \mathbb{R},
\end{align}
and $j_i$ is obtained from $\theta_i$ by the following relation 
\begin{align}\label{eqn-310}
j_i(x,t,\xi)= \int_0^\xi \theta_i(x,t,s)\d s.
\end{align}

The following result guarantees the measurability of the lower and upper envelopes $\underline{\theta_i}(x,t,\xi)$ and 
$\overline{\theta_i}(x,t,\xi)$, respectively.
\begin{proposition}[{\cite[Proposition 3.1]{SMAO}}]
	Assume that the function $\theta_i$ satisfies Hypothesis \ref{hyptheta}. Then, the functions
	$$(x,t) \mapsto \underline{\theta}_i(x,t,\u_i(x,t)) \ \text{ and }\ (x,t) \mapsto \overline{\theta}_i(x,t,\u_i(x,t))$$
	are measurable on $\mathfrak{D} \times(0,T)$ for any measurable function $\u_i\colon \mathfrak{D} \times(0,T) \to \mathbb{R}$ for $i=1,\ldots,d$.
\end{proposition} 
%The existence of solution will be proved by applying the Galerkin method to a regularized problem which is given below. 
To define the regularized problem, we consider a mollifier 
\begin{align}\label{uprho-d}
\uprho \in \mathrm{C}_0^{\infty}(-1,1), ~\uprho \geq 0,~ \int_{\mathbb{R}}\uprho(s)\d s=1.
\end{align}
The existence of such mollifier is ensured from \cite[Lemma 7.3.1, pp. 109] {DGDM}. For $\varepsilon>0$, we define $\uprho_\varepsilon(s) := \frac{1}{\varepsilon}\uprho\left(\frac{s}{\varepsilon}\right)$. Now, we introduce the regularization
 \begin{align}\label{Reg-1}
 	 \theta_{i,\varepsilon}(x,t,\xi) := \uprho_{\varepsilon} * \theta_i(x,t,\xi)= \int_{\mathbb{R}}\uprho_\varepsilon(s)\theta_i(x,t,\xi-s)\d s=\int_{\mathbb{R}}\uprho_\varepsilon(\xi-s)\theta_i(x,t,s)\d s.
 \end{align}
The following result (see \cite[Lemma 7.3.2, pp. 109-110]{DGDM}) plays a crucial role in establishing the energy estimates for the finite-dimensional regularized problem arising in the Galerkin approximation procedure (see Step II of Theorem \ref{thm 4.3}). 
%For  reader's convenience, we present a proof of this result below.
\begin{lemma}\label{lemma-1}
Let $\uprho$ be the mollifier defined in \eqref{uprho-d}. Suppose that the functions $\theta_i$ satisfy condition \eqref{thetsupinf}. Then, there exist positive constants $c_{i1}, c_{i2}>0, \ i=1,\ldots,d,$ such that for sufficiently small $\varepsilon >0$, the following assertions hold:
\begin{align*}
\sum_{i=1}^{d}\int_{\mathfrak{D}} \theta_{i, \varepsilon}(x,t,\xi)\xi\d x \geq -\sum_{i=1}^{d}c_{i1}c_{i2}|\mathfrak{D}|, \ \text{ for all } \
\xi \in \R,
\end{align*}
where the regularized functions $\theta_{i,\varepsilon}$ satisfy 
\begin{equation*}
	\left\{
	\begin{aligned}
	\theta_{i, \varepsilon}(x,t,\xi) \leq 0 &\ \text{ if }\  \xi < -c_{i1},\\
	\theta_{i, \varepsilon}(x,t,\xi) \geq 0 &\ \text{ if } \ \xi > c_{i1}
	\end{aligned}
	\right.
\end{equation*}
and 
\begin{align*}
 |\theta_{i, \varepsilon}(x,t,\xi)| \leq c_{i2} \ \text{ if } \ |\xi| \leq c_{i1}.
\end{align*}
\end{lemma}
Let us now state and prove the existence of weak solutions to Problem \ref{prob-inequality} defined in the sense of Definition \ref{weakdef}.
\begin{theorem}\label{thm 4.3}
Let  $r,q\geq 1$ with $r>q$. Assume that $\theta_i$ satisfies the Hypothesis \ref{hyptheta} for each $i=1,\ldots, d$. Then, for any $\f\in \mathrm{L}^2(0,T;\V^{\prime})$ and $\u^0 \in \H$, the Problem \ref{prob-inequality} admits a weak solution in the sense of Definition \ref{weakdef}.
\end{theorem}
\begin{proof}
We employ the Galerkin method to establish the existence of a solution to Problem \ref{prob-inequality}. The proof proceeds rigorously through the following number of steps:
\vskip 0.2cm
\noindent
\textbf{Step I:} \emph{Regularized finite-dimensional problem.} Let  $\{\w_j\}_{j\in\N}$ be a basis in $\V$, that is, $\{\w_j\}_{j\in\N}$ forms a countable sequence of dense elements in $\V$, finitely many $\{\w_j:j=1,\ldots,n, \ n\in\N\}$ are linearly independent. Since $\V$ is separable, the existence of such basis is guaranteed. By using an Gram-Schmidt orthogonalization process, we may assume that $\{\w_j\}_{j\in\N}$ is an orthonormal basis in $\H$. Let us denote by $\H_n:=\mathrm{span}\{\w_j:j=1,\ldots,n, \ n\in\N\}$, with the norm inherited from $\H$. Let $\u^0_n$ be  the orthogonal projection in $\H$ of $\u^0$ onto the space $\H_n$ such that 
\begin{align*}
	\u^0_n \to \u^0 \text{ in } \H, \text{ as } n \to \infty.
\end{align*}
Let $\{\varepsilon_n\}$ be a sequence of positive numbers converging decreasingly to $0$. For the sake of simplicity, we take $\theta_{i, \varepsilon_n} \equiv \theta_{in} \text{ and } \theta_{in}(x,t,\u_{n,i}(x,t)) \equiv \theta_{in}(t,\u_{n,i}(t))$. Let us now fix $T>0$ and for each $n\in\N$, we search for a approximate solution of the form 
\begin{align*}
	\u_n(x,t):=\sum\limits_{k=1}^n g_k^n(t)\w_k(x),
\end{align*}
such that
%where $g_1^n(t),\ldots,g_k^n(t)$ are unknown scalar functions of $t$ such that it satisfies the following finite-dimensional system of ordinary differential equations in $\H_n$:
\begin{equation}\label{appxode1}
	\left\{
	\begin{aligned}
		\frac{\d}{\d t}\left(\u_n(t),\v\right)&+ (\mu\A\u_n(t)+\B(\u_n(t))+\alpha\wi{\C}(\u_n(t))+\beta\C(\u_n(t)),\v)\\
		&+\sum_{i=1}^{d}\int_{\mathfrak{D}} \theta_{in}(t, \u_{n,i}(t)) \v_i \d x= \langle {\f}(t),\v \rangle, \text{ a.e. } t\in[0,T] \text{ and all } \v \in \H_n, \\
		\u_n(0)&=\u_{n}^0.
	\end{aligned}
	\right.
\end{equation}
That is
\begin{equation}\label{appxode2}
	\left\{
	\begin{aligned}
\sum\limits_{k=1}^n(g_k^n)^{\prime}(t)\left(\w_k,\w_j\right)+\mu\sum\limits_{k=1}^n g_k^n(t)(\A\w_k,\w_j)+
\bigg(\B\big(\sum\limits_{k=1}^n g_k^n(t)&\w_k\big) +\alpha\wi{\C}\big(\sum\limits_{k=1}^n g_k^n(t)\w_k\big) \\+\beta\C\big(\sum\limits_{k=1}^n g_k^n(t)\w_k\big),\w_j\bigg) +\sum_{i=1}^{d}\int_{\mathfrak{D}}\theta_{in}(t, \u_{n,i}(t)) \w_j \d x&= \langle {\f}(t),\w_j\rangle, \ j=1,\ldots,n ,\\
   g_j^n(0)&=\u_{n,j}^0,\ j = 1,\ldots, n,
\end{aligned}
		\right.
	\end{equation}
where $\u_{n,j}^0$ is the $j$-th component of $\u^0_n$ and $g_k^n(t)$, $k=1,\ldots,n$, are unknown scalar functions of $t$. The system \eqref{appxode2} represents a finite-dimensional system of ordinary differential equations in $\H_n$.
Since the operators $\B(\cdot)$, $\C(\cdot)$ and $\wi{\C}(\cdot)$ are locally Lipschitz (see \eqref{lip} and \eqref{213}), an application of Carath\'eodory's existence theorem \cite[pp. 1044]{EZ} guarantees the existence of a local maximal solution $\u_n\in\mathrm{C}([0,T^*];\H_n),$  for some $0<T^*\leq T$ to the system \eqref{appxode2}. The uniqueness of this solution follows directly from the local Lipschitz property. 
\vskip 0.2cm
\noindent
\textbf{Step II:} \emph{A priori estimates for solutions of the Galerkin equation \eqref{appxode1}.}
%	Let $\u_0\in \H$ and $\f\in\mathrm{L}^2(0,T;\V^{\prime})$  be given. Then, for $r\in[1,\infty)$, we have 
%	\begin{align}\label{energy1}
%		\sup_{t\in[0,T]}\|\u^n(t)\|_{\H}^2+\mu \int_0^T\|\u^n(t)\|_{\V}^2\d t+2\beta\int_0^T\|\u^n(t)\|_{\widetilde\L^{r+1}}^{r+1}\d t\leq \|\u_0\|_{\H}^2+\frac{1}{\mu }\int_0^T\|\f(t)\|_{\V'}^2\d t.
%	\end{align}
The time $T^*$ can be extended to $T$ by establishing uniform energy estimates for the solution $\u_n$ of the system \eqref{appxode1}.
From the definition of the operators $\A$, $\B(\cdot)$, $\C(\cdot)$ and $\wi{\C}(\cdot)$, it is observed that
	\begin{align*}
%		(\A_n\u^n,\u^n)&=(\P_n\A\u^n,\u^n)
		(\A\u_n,\u_n)=\|\u_n\|_{\V}^2,
%		(\B_n(\u^n),\u^n)&=(\P_n\B(\u^n),\u^n)=
		(\B(\u_n),\u_n)=0,
%		(\C_n(\u^n),\u^n)&=(\P_n\C(\u^n),\u^n)=
		(\C(\u_n),\u_n)=\|\u_n\|_{\widetilde\L^{r+1}}^{r+1},
	(\wi{\C}(\u_n),\u_n)=\|\u_n\|_{\widetilde\L^{q+1}}^{q+1}.
	\end{align*}
For $\v=\w_k$, multiplying \eqref{appxode1} by $g_k^n(\cdot)$, summing up over $k=1,\ldots,n$ and using the above equations, we obtain
	\begin{align*}
	&\frac{1}{2}\frac{\d}{\d t}\|\u_n(t)\|_{\H}^2+\mu \|\u_n(t)\|_{\V}^2+\alpha\|\u_n(t)\|_{\widetilde\L^{q+1}}^{q+1}+\beta\|\u_n(t)\|_{\widetilde\L^{r+1}}^{r+1}\\
	&\quad+\sum_{i=1}^{d}\int_{\mathfrak{D}} \theta_{in}(t, \u_{n,i}(t)) \u_{n,i}(t) \d x=\langle {\f}(t),\u_n(t)\rangle,
	\end{align*}
	for a.e. $t\in[0,T]$. 
	%	where we performed an integration by parts and used the fact that $(\B(\u^n),\u^n)=0$.
	Integrating the equality from $0$ to $t$, we find 
	\begin{align}\label{415}
		&\|\u_n(t)\|_{\H}^2+2\mu \int_0^t\|\u_n(s)\|_{\V}^2\d s+ 2\beta\int_0^t\|\u_n(s)\|_{\widetilde\L^{r+1}}^{r+1}\d s\nonumber\\&= \|\u_n(0)\|_{\H}^2+2 \int_0^t\langle{\f}(s),\u_n(s)\rangle\d s- 2\alpha\int_0^t\|\u_n(s)\|_{\widetilde\L^{q+1}}^{q+1}\d s \nonumber\\&\quad-2\int_0^t\sum_{i=1}^{d}\int_{\mathfrak{D}} \theta_{in}(s, \u_{n,i}(s)) \u_{n,i}(s)\d x \d s,
	\end{align}
	for all $t\in[0,T]$. Using the Cauchy-Schwarz and Young inequalities, we estimate $|\langle{\f},\u_n\rangle|$ as 
	\begin{align}\label{ineq-2}
		|\langle{\f},\u_n\rangle|\leq \|{\f}\|_{\V^{\prime}}\|\u_n\|_{\V}
		\leq\frac{1}{2\mu }\|{\f}\|_{\V^{\prime}}^2+\frac{\mu }{2}\|\u_n\|_{\V}^2.
	\end{align}
	Now, using H\"older's and Young's inequalities with exponents $\frac{r+1}{q+1}$ and  $\frac{r+1}{r-q}$, we estimate 
\begin{align}\label{ineq-3}
	-2\alpha\|\u_n\|_{\widetilde\L^{q+1}}^{q+1}\leq 2|\alpha| (|\mathfrak{D}|)^{\frac{r-q}{r+1}}\|\u_n\|_{\widetilde\L^{r+1}}^{q+1} \leq \kappa (2|\alpha|)^{\frac{r+1}{r-q}}|\mathfrak{D}|+ \beta \|\u_n\|_{\widetilde\L^{r+1}}^{r+1}, 
\end{align}
where $\kappa= \left(\frac{q+1}{\beta(r+1)}\right)^{\frac{q+1}{r-q}}\left(\frac{r-q}{r+1}\right)$. Thus, using Lemma \ref{lemma-1}, and inequalities \eqref{ineq-2}-\eqref{ineq-3}, we obtain from \eqref{415}  that 
\begin{align}\label{estimate-1}
	&\|\u_n(t)\|_{\H}^2+\mu \int_0^t\|\u_n(s)\|_{\V}^2\d s +\beta\int_0^t\|\u_n(s)\|_{\widetilde\L^{r+1}}^{r+1}\d s\nonumber\\
		&\leq \|\u^0\|_{\H}^2+\frac{1}{\mu }\int_0^t\|\f(s)\|_{\V^{\prime}}^2\d s+\kappa (2|\alpha|)^{\frac{r+1}{r-q}}|\mathfrak{D}|T+ \sum_{i=1}^{d}2c_{i1}c_{i2}|\mathfrak{D}|T,
	\end{align}	
for all $t\in[0,T]$. 
\vskip 0.2cm
\noindent
\emph{ Boundedness of the sequence 
	$\{\theta_{in}(\cdot, \u_{n,i}(\cdot))\}_{n\in\N}$.} 
	From \eqref{growth}, we find
		\begin{align}\label{reg-est-3}
\int_{0}^{T}\int_{\mathfrak{D}}|\theta_{in}(s, \u_{n,i}(s))|^2\d x\d s
&\leq \int_{0}^{T}\int_{\mathfrak{D}} (\bar{\alpha}(x,t)+C_{1,i} + C_{2,i}|\u_{n,i}(s)|)^2 \d x\d s \nonumber\\
	&\leq 3\int_{0}^{T}\int_{\mathfrak{D}} (|\bar{\alpha}(x,t)|^2+C_{1,i}^2+ C_{2,i}^2|\u_{n,i}(s)|^2) \d x\d s\nonumber\\
	&\leq 3(C_0+ C_{1,i}^2|\mathfrak{D}|T+C_{2,i}^2\|\u_{n,i}\|_{\mathrm{L}^2(0,T;\H)}^2),
		\end{align}
where $C_0=\int_{0}^{T}\int_{\mathfrak{D}} |\bar{\alpha}(x,t)|^2\d x\d t$. Finally, from \eqref{reg-est-3}, it follows that the sequence $\{\theta_{in}(\cdot, \u_{n,i}(\cdot))\}_{n\in\N}$ is bounded in $\mathrm{L}^2(0,T;\H).$  
\vskip 0.2cm
\noindent
\emph{ Boundedness of the sequence 
	$\{\mu\A\u_n+\B(\u_n)+\alpha\wi{\C}(\u_n)+\beta \C(\u_n)\}_{n\in\N}$.}
We next claim that the sequence $\{\mathscr{F}(\u_n)\}_{n\in\N}$ is uniformly bounded in $n$ for all $r\geq1$, where 
\begin{align}\label{eqn-nem-f}
	\mathscr{F}(\cdot):=\mu\A+\B(\cdot)+\alpha\wi{\C}(\cdot)+\beta \C(\cdot).
\end{align}
\vskip 2mm
\noindent
From application of H\"older's inequality and \eqref{estimate-1} yields 
\begin{align}\label{Conver-1}
	&\left| \int_0^T \langle \mathscr{F}(\u_n(t)), \boldsymbol{\Psi}(t) \rangle \d t \right| \nonumber\\
&\leq \mu \left|\int_0^T (\nabla \u_n(t), \nabla \boldsymbol{\Psi}(t)) \d t \right|
	+ \left| \int_0^T \langle \B(\u_n(t), \u_n(t)),\boldsymbol{\Psi}(t)\rangle \d t \right|\nonumber\\
	&\quad+ |\alpha| \left| \int_0^T \langle |\u_n(t)|^{q-1} \u_n(t),\boldsymbol{\Psi}(t) \rangle \d t \right|+\beta \left| \int_0^T \langle |\u_n(t)|^{r-1} \u_n(t),\boldsymbol{\Psi}(t) \rangle \d t \right|\nonumber\\
	&\leq \mu \int_0^T \|\nabla \u_n(t)\|_{\H} \|\nabla\boldsymbol{\Psi}(t)\|_{\H} \d t
	+ \int_0^T \|\u_n(t)\|_{\widetilde\L^{4}}
	\|\u_n(t)\|_{\V}
	\|\boldsymbol{\Psi}(t)\|_{\widetilde\L^{4}}\d t \nonumber\\
	&\quad+|\alpha| \int_0^T \|\u_n(t)\|^{q}_{\widetilde\L^{q+1}} \|\boldsymbol{\Psi}(t)\|_{\widetilde\L^{q+1}} \d t+ \beta \int_0^T \|\u_n(t)\|^{r}_{\widetilde\L^{r+1}} \|\boldsymbol{\Psi}(t)\|_{\widetilde\L^{r+1}} \d t \nonumber\\
	&\leq \mu  \|\u_n\|_{\mathrm{L}^2(0,T;\V)}\|\boldsymbol{\Psi}\|_{\mathrm{L}^2(0,T;\V)}
	+ \|\u_n\|_{\mathrm{L}^{4}(0,T;\wi\L^{4})}
	\|\u_n\|_{\mathrm{L}^{2}(0,T;\V)}\|\boldsymbol{\Psi}\|_{\mathrm{L}^{4}(0,T;\wi\L^{4})}
	\nonumber\\
	&\quad+\bigg(|\alpha|\|\u_n\|_{\mathrm{L}^{r+1}(0,T;\widetilde\L^{r+1})}^q+ \beta \|\u_n\|_{\mathrm{L}^{r+1}(0,T;\widetilde\L^{r+1})}^r \bigg) \|\boldsymbol{\Psi}\|_{\mathrm{L}^{r+1}(0,T;\widetilde\L^{r+1})} \nonumber\\
	&\leq C\big( \|\u^0\|_{\H}, \mu, T, |\alpha|, \beta, \|\f\|_{\mathrm{L}^2(0,T;\V^{\prime})} \big)
	\left(\|\boldsymbol{\Psi}\|_{\mathrm{L}^2(0,T;\V)}
	+ \|\boldsymbol{\Psi}\|_{\mathrm{L}^{r+1}(0,T;\widetilde\L^{r+1})}+\|\boldsymbol{\Psi}\|_{\mathrm{L}^{4}(0,T;\wi\L^{4})}\right),
\end{align}
for all $\boldsymbol{\Psi} \in \mathrm{L}^2(0,T;\V) \cap \mathrm{L}^{r+1}(0,T;\widetilde\L^{r+1}) \cap \mathrm{L}^{4}(0,T;\wi\L^{4})$. Note that for $r\in[1,3]$, the above inequality \eqref{Conver-1} holds true for all $\boldsymbol{\Psi} \in \mathrm{L}^2(0,T;\V) \cap \mathrm{L}^{4}(0,T;\wi\L^{4})$ and for $r>3$, the inequality \eqref{Conver-1} holds true for all $\boldsymbol{\Psi} \in \mathrm{L}^2(0,T;\V) \cap \mathrm{L}^{r+1}(0,T;\widetilde\L^{r+1})$.
\vskip 0.2cm
\noindent
 \emph{Boundedness of the derivative term $\frac{\d\u_n}{\d t}$.}
Using H\"older's inequality along with \eqref{reg-est-3}, \eqref{Conver-1}, we find
\begin{align}\label{derbdd2}
	&\left| \int_0^T \left\langle \frac{\d \u_n}{\d t},\boldsymbol{\Psi}(t) \right\rangle \d t \right| \nonumber\\
	&\leq \left| \int_0^T \langle \mathscr{F}(\u_n(t)),\boldsymbol{\Psi}(t) \rangle \d t \right|
	+\left| \int_0^T \langle \f(t),\boldsymbol{\Psi}(t) \rangle \d t \right|+ \left| \int_0^T\sum_{i=1}^{d}\int_{\mathfrak{D}} \theta_{in}(t, \u_{n,i}(t))\boldsymbol{\Psi}(t) \d x \d t \right|\nonumber\\
	&\leq \left| \int_0^T \langle \mathscr{F}(\u_n(t)),\boldsymbol{\Psi}(t) \rangle \d t \right|
	+ \left| \int_0^T \langle \f(t),\boldsymbol{\Psi}(t) \rangle \d t \right| + \sum_{i=1}^{d} \int_0^T \|\theta_{in}(t, \u_{n,i}(t))\|_{\H} \|\boldsymbol{\Psi}(t)\|_{\H} \d t \nonumber\\
	&\leq C\big( \|\u^0\|_{\H}, \mu, T, |\alpha|, \beta, \|\f\|_{\mathrm{L}^2(0,T;\V^{\prime})} \big) \left(\|\boldsymbol{\Psi}\|_{\mathrm{L}^2(0,T;\V)}
	+ \|\boldsymbol{\Psi}\|_{\mathrm{L}^{r+1}(0,T;\widetilde\L^{r+1})}+\|\boldsymbol{\Psi}\|_{\mathrm{L}^{r+1}(0,T;\wi\L^{r+1})}\right)\nonumber\\
	&\quad + \|\f\|_{\mathrm{L}^2(0,T;\V^{\prime})} \|\boldsymbol{\Psi}\|_{\mathrm{L}^2(0,T;\V)} + C \sum_{i=1}^{d} \|\theta_{in}(\cdot, \u_{n,i}(\cdot))\|_{\mathrm{L}^2(0,T;\H)}\|\boldsymbol{\Psi}\|_{\mathrm{L}^2(0,T;\V)} \nonumber\\
	&\leq C\big( \|\u^0\|_{\H}, \mu, T, |\alpha|, \beta, \|\f\|_{\mathrm{L}^2(0,T;\V^{\prime})},|\mathfrak{D}| \big)
	\nonumber\\&\qquad\times
	 \left(\|\boldsymbol{\Psi}\|_{\mathrm{L}^2(0,T;\V)}
	+ \|\boldsymbol{\Psi}\|_{\mathrm{L}^{r+1}(0,T;\widetilde\L^{r+1})}+
	\|\boldsymbol{\Psi}\|_{\mathrm{L}^{r+1}(0,T;\wi\L^{r+1})}\right),
\end{align}
for all $ \boldsymbol{\Psi} \in \mathrm{L}^2(0,T;\V) \cap \mathrm{L}^{r+1}(0,T;\widetilde\L^{r+1}) \cap \mathrm{L}^{4}(0,T;\wi\L^{4})$. 
	\vskip 0.2cm
	\noindent
   \textbf{Step III:}\emph{ Extract weak convergent subsequences}. 
%  We know that the dual of $\mathrm{L}^1(0, T; \H)$ is $\mathrm{L}^{\infty}(0, T; \H)$, that is, $(\mathrm{L}^1(0, T; \H))^{\prime} \cong \mathrm{L}^{\infty}(0, T; \H)$ and the space $\mathrm{L}^1(0, T; \H)$ is separable. Furthermore, the spaces $\mathrm{L}^2(0, T; \V) \text{ and } \mathrm{L}^{r+1}(0, T; \wi\L^{r+1})$ are reflexive. Therefore,
   In view of the uniform energy estimates \eqref{estimate-1}-\eqref{reg-est-3}, 
   \eqref{Conver-1} and \eqref{derbdd2}, along with the Banach-Alaoglu theorem, we have following weak and weak$^*$ convergences:
   \begin{equation}\label{wconvergence}
   	\left\{
   	\begin{aligned}
   	&\u_n \xrightarrow{w^*} \u\ \text{ in }\ \mathrm{L}^{\infty}(0,T;\H), \\
   	&\u_n \xrightarrow{w} \u\ \text{ in }\ \mathrm{L}^{2}(0,T;\V), \\
   	& \u_n \xrightarrow{w} \u\ \text{ in }\ \mathrm{L}^{r+1}(0,T;\widetilde\L^{r+1}), 
   	\end{aligned}
   	\right.
   	\left\{
   	\begin{aligned}
   	& \theta_{in}(\cdot, \u_{n,i}(\cdot)) \xrightarrow{w} \chi_i \text{ in }\ \mathrm{L}^{2}(0,T;\H),\\
   	&\mathscr{F}(\u_n)\xrightarrow{w}\mathscr{F}_0\ \text{ in }\ \mathrm{L}^{2}(0,T;\V^{\prime})+
   	\mathrm{L}^{\frac{r+1}{r}}(0,T;\widetilde\L^\frac{r+1}{r} ),\\
   	&	\frac{\d \u_n}{\d t} \xrightarrow{w} \wi\u \ \text{ in }\ 
   	\mathrm{L}^{2}(0,T;\V^{\prime})+\mathrm{L}^{\frac{r+1}{r}}(0,T;\widetilde\L^{\frac{r+1}{r}}).
   	\end{aligned}\right.
   \end{equation}
 Since 
 \begin{align*}
 {\u}_n &\xrightarrow{w} \u \text{ in } \mathrm{L}^{2}(0,T;\V),\\
 \text{ and } \ 
 \frac{\d \u_n}{\d t} &\xrightarrow{w} \wi\u \ \text{ in } 	\mathrm{L}^{2}(0,T;\V^{\prime})+\mathrm{L}^{\frac{r+1}{r}}(0,T;\widetilde\L^{\frac{r+1}{r}})\hookrightarrow\mathrm{L}^{\frac{r+1}{r}}(0,T;
 \V^{\prime}+\wi\L^{\frac{r+1}{r}})
 \end{align*}
   and the embedding $\V\hookrightarrow\V^{\prime}+\wi\L^{\frac{r+1}{r}}$ is continuous, it follows from \cite[Proposition 1.2, Chapter 1]{JHMMPD} that 
\begin{align*}
	\wi\u=\frac{\d \u}{\d t} \ \text{ for a.e. } \ t\in[0,T].
\end{align*}
Moreover, by an application of the Aubin-Lions lemma (see Theorem \ref{thm-Simon}) and the convergences 
\begin{align*}
{\u}_n \xrightarrow{w} \u \text{ in }
\mathrm{L}^{2}(0,T;\V) \ \text{ and } \ 
\frac{\d \u_n}{\d t} \xrightarrow{w} \frac{\d \u}{\d t} \text{ in } \mathrm{L}^{\frac{r+1}{r}}(0,T;
\V^{\prime}+\wi\L^{\frac{r+1}{r}}),
\end{align*}
yields the following strong convergence:
\begin{align}\label{3p58}
	\u_n&\rightarrow\u\ \text{ in }\ \mathrm{L}^{2}(0,T;\H).
\end{align}
	\vskip 0.2cm
\noindent
\textbf{Step IV:}\emph{ Passing limit as $n\to\infty $ in \eqref{appxode1}}. Let us now discuss the convergence of the operator terms (linear, bilinear and non-linear) involved in the equation \eqref{appxode1}.
\vskip 0.2cm
\noindent
\emph{Convergence of the linear term:} 
Since, the operator $\A$ is linear and continuous from $\mathrm{L}^2(0,T;\V)$ to $\mathrm{L}^2(0,T;\V^{\prime})$, it is in particular weakly continuous. Therefore, from the convergence 
$$\u_n\xrightarrow{w}\u \text{ in } \mathrm{L}^{2}(0,T;\V),$$ 
it follows immediately that for every $\z\in\mathrm{L}^2(0,T;\V)$
\begin{align*}
\int_0^T\langle\A\u_n(t),\z(t)\rangle\d t \to \int_0^T\langle\A\u(t),\z(t)\rangle \d t \text{ as }\ n\to\infty.
\end{align*}
\vskip 0.2cm
\noindent
\emph{Convergence of the bilinear term:} 
Let us take $\z\in\mathrm{C}([0,T];\V)$. Now, by using Ladyzhenskaya's and H\"older's inequalities, estimate \eqref{estimate-1}, and the convergence \eqref{3p58}, we calculate 
\begin{align}
	&\left|\int_0^T\langle\B(\u_n(t)),\z(t)\rangle\d t- \int_0^T\langle\B(\u(t)),\z(t)\rangle\d t\right|
	\nonumber\\
	&\leq\left|\int_0^T\langle\B(\u_n(t)-\u(t),\u_n(t)),\z(t)\rangle\d t\right|+\left|\int_0^T\langle\B(\u(t),\u_n(t)-\u(t)),\z(t)\rangle\d t\right| 
	\nonumber\\
	&=\left|\int_0^T\langle\B(\u_n(t)-\u(t),\z(t)),\u_n(t)\rangle\d t\right|+\left|\int_0^T\langle\B(\u(t),\z(t)),\u_n(t)-\u(t)\rangle\d t\right| \nonumber\\
	&\leq \int_{0}^{T} \|\u_n(t)-\u(t)\|_{\widetilde\L^4}\|\z(t)\|_{\V}\|\u_n(t)\|_{\widetilde\L^4} \d t+\int_{0}^{T} \|\u(t)\|_{\widetilde\L^4}\|\z(t)\|_{\V}\|\u_n(t)-\u(t)\|_{\widetilde\L^4} \d t
	 \nonumber\\ 
	 &\leq \sup_{t\in[0,T]}\|\z(t)\|_{\V} \left(\int_{0}^{T} \|\u_n(t)-\u(t)\|_{\widetilde\L^4}^2 \d t \right)^{\frac{1}{2}} \big[\|\u_n\|_{\mathrm{L}^2(0,T;\wi\L^4)}+
	 \|\u\|_{\mathrm{L}^2(0,T;\wi\L^4)}\big]  \nonumber\\ 
	 &\leq C\sup_{t\in[0,T]}\|\z(t)\|_{\V} \left(\int_{0}^{T} \|\u_n(t)-\u(t)\|_{\H}^{2\left(\frac{4-d}{4}\right)} \|\u_n(t)-\u(t)\|_{\V}^{\frac{2d}{4}} \d t \right)^{\frac{1}{2}} \nonumber\\
	 &\qquad\times\big[\|\u_n\|_{\mathrm{L}^2(0,T;\V)}+
	 \|\u\|_{\mathrm{L}^2(0,T;\V)}\big]
	 \nonumber\\&\leq 
	 2^{\frac{d}{4}}C\sup_{t\in[0,T]}\|\z(t)\|_{\V} \|\u_n-\u\|_{\mathrm{L}^2(0,T;\H)}^{\frac{4-d}{4}}
	  \big[\|\u_n\|_{\mathrm{L}^2(0,T;\V)}+
	  \|\u\|_{\mathrm{L}^2(0,T;\V)}\big]^{\frac{d}{4}+1}
	  \label{bunmbu}\\
	 &\to0\ \text{ as }\ n\to\infty.\nonumber 
\end{align}
Thus, we obtain for all $\z\in\mathrm{C}([0,T];\V),$
\begin{align}\label{eqn-conv-b-3}
\int_0^T\langle\B(\u_n(t)),\z(t)\rangle\d t\to
\int_0^T\langle\B({\u}(t)),\z(t)\rangle\d t  \ \text{ as } \ n\to\infty.
\end{align}
Since $\mathrm{C}([0,T];\V)$ is dense in $\mathrm{L}^{\frac{4}{4-d}}(0,T;\V)$, one can prove that the above convergence holds true for all $\z\in \mathrm{L}^{\frac{4}{4-d}}(0,T;\V)$ also. For this, assume $\z\in \mathrm{L}^{\frac{4}{4-d}}(0,T;\V)$. Then, for any given $\varepsilon>0$, there exists $\z_{\varepsilon}\in \mathrm{C}([0,T];\V)$ such that $$\|\z_{\varepsilon}-\z\|_{\mathrm{L}^{\frac{4}{4-d}}(0,T;\V)}\leq {\varepsilon}.$$ 
Similar to \eqref{bunmbu}, we compute
\begin{align}\label{3.33}
	&\left|\int_0^T\langle\B(\u_n(t))-\B({\u}(t)),\z(t)\rangle
	\d t\right|
	\nonumber\\
	&\leq		
	\left|\int_0^T\langle\B(\u_n(t)-\B({\u}(t)),\z(t)-\z_\eps(t)\rangle 
	\d t\right|+
	\left|\int_0^T\langle\B(\u_n(t))-\B({\u}(t)),\z_\eps(t)\rangle
	\d t\right|
	\nonumber\\
	&\leq \int_0^T\|\u_n(t)\|_{\wi\L^4}^{2}\|\z(t)-\z_{\varepsilon}(t)\|_{\V} \d t+\int_0^T\|\u(t)\|_{\wi\L^4}^{2}\|\z(t)-\z_{\varepsilon}(t)\|_{\V} \d t \nonumber\\
	&\quad+
	\left|\int_0^T\langle\B(\u_n(t))-\B({\u}(t)),\z_\eps(t)\rangle
	\d t\right| 	\nonumber\\
	&\leq 	C\left[\left(\int_0^T\|\u_n(t)\|_{\wi\L^4}^{\frac{8}{d}}\d t\right)^{\frac{d}{4}}+\left(\int_0^T\|\u(t)\|_{\wi\L^4}^{\frac{8}{d}}\d t\right)^{\frac{d}{4}}\right]\left(\int_0^T\|\z(t)-\z_{\varepsilon}(t)\|_{\V}^{\frac{4}{4-d}}\d t\right)^{\frac{4-d}{4}} \nonumber\\
	&\quad+
	\left|\int_0^T\langle\B(\u_n(t))-\B({\u}(t)),\z_\eps(t)\rangle
	\d t\right|.
\end{align}
Since an application of Ladyzheskaya's inequality and the energy estimate \eqref{estimate-1} yield
\begin{align}\label{3.34}
	\int_0^T\|\u_n(t)\|_{\wi\L^4}^{\frac{8}{d}}\d t&\leq C\sup_{t\in[0,T]}\|\u_n(t)\|_{\H}^{\frac{2(4-d)}{d}}\int_0^T\|\u_n(t)\|_{\V}^2\d t\nonumber\\&\leq \left\{ \|\u^0\|_{\H}^2+\frac{1}{\mu }\int_0^t\|\f(s)\|_{\V^{\prime}}^2\d s+\kappa (2|\alpha|)^{\frac{r+1}{r-q}}|\mathfrak{D}|T+ \sum_{i=1}^{d}2c_{i1}c_{i2}|\mathfrak{D}|T\right\}^{\frac{4}{d}}.
\end{align}
Thus, by using \eqref{3.34} and \eqref{eqn-conv-b-3} in \eqref{3.33}, we obtain
\begin{align*}
	&\lim_{n\to\infty}	\left|\int_0^T\langle\B(\u_n(t)),\z(t)\rangle\d t-\int_0^T\langle\B(\u(t)),\z(t)\rangle\d t\right|\nonumber\\&\leq C\left\{ \|\u^0\|_{\H}^2+\frac{1}{\mu }\int_0^t\|\f(s)\|_{\V^{\prime}}^2\d s+\kappa (2|\alpha|)^{\frac{r+1}{r-q}}|\mathfrak{D}|T+ \sum_{i=1}^{d}2c_{i1}c_{i2}|\mathfrak{D}|T \right\}\varepsilon.
\end{align*}
Since $\varepsilon>0$ is arbitrary, we deduce that 
\begin{align*}
	\lim_{n\to\infty}\int_0^T\langle\B(\u_n(t)),\z(t)\rangle\d t=\int_0^T\langle\B(\u(t)),\z(t)\rangle\d t,
\end{align*}
for any $\z\in \mathrm{L}^{\frac{4}{4-d}}(0,T;\V)$.
\vskip 0.2cm
\noindent
\emph{Convergence of the nonlinear term:} We know from \eqref{estimate-1} that 
\begin{align}\label{3p57}
	\int_0^T	\|\C(\u_n(t))\|_{\wi\L^{\frac{r+1}{r}}}^{\frac{r+1}{r}}dt &\leq  \int_0^T\|\u_n(t)\|_{\wi\L^{r+1}}^{r+1}\d t \nonumber\\ &\leq \|\u^0\|_{\H}^2+\frac{1}{\mu }\int_0^t\|\f(s)\|_{\V^{\prime}}^2\d s+\kappa (2|\alpha|)^{\frac{r+1}{r-q}}|\mathfrak{D}|T+ \sum_{i=1}^{d}2c_{i1}c_{i2}|\mathfrak{D}|T.
\end{align}
Therefore, by an application of the Banach-Alaoglu theorem, there exists $\boldsymbol{\eta}\in\mathrm{L}^{\frac{r+1}{r}}(0,T;\wi\L^{\frac{r+1}{r}})$ such that 
\begin{align*}
	\C(\u_n)\xrightarrow{w} \boldsymbol{\eta}\ \text{ in }\ \mathrm{L}^{\frac{r+1}{r}}(0,T;\wi\L^{\frac{r+1}{r}}) .
\end{align*}
Note that from \eqref{3p57}, we have $\{|\u_n(t)|^{r-1}\u_n(t)\}_{n\in\N}, |\u(t)|^{r-1}\u(t) \in \mathrm{L}^{\frac{r+1}{r}}((0,T)\times \mathfrak{D})$ and 
 from the convergence \eqref{3p58}, we can extract a further subsequence (labeled with the same symbol) such that  
\begin{align}\label{3p61}
	\u_n(x,t)\to \u(x,t)\ \text{ for a.e. }\ (x,t)\in \mathfrak{D} \times (0,T). 
\end{align}
Hence, \eqref{3p61} implies that $|\u_{n}(t)|^{r-1}\u_{n}(t) \to |\u(t)|^{r-1}\u(t)$ a.e. in $\mathfrak{D} \times (0,T)$. According to Lemma \ref{Lem-Lions} and \eqref{3p57}, we deduce
\begin{align*}
	|\u_{n}(t)|^{r-1}\u_{n}(t) \xrightarrow{w} |\u(t)|^{r-1}\u(t) \  \mbox{ in } \  \mathrm{L}^{\frac{r+1}{r}}(\mathfrak{D} \times (0,T)).
\end{align*} 
Thus, by the uniqueness of weak limits, we deduce that $\boldsymbol{\eta}=|\u|^{r-1}\u$, which implies that
\begin{align*}
	\C(\u_{n}) \xrightarrow{w} \C(\u) \  \mbox{ in } \  \mathrm{L}^{\frac{r+1}{r}}(0,T;\wi\L^{\frac{r+1}{r}}).
\end{align*}
Analogously, one can also deduce the following:
\begin{align*}
	\wi\C(\u_{n}) \xrightarrow{w} \wi\C(\u) \  \mbox{ in } \  \mathrm{L}^{\frac{q+1}{q}}(0,T;\wi\L^{\frac{q+1}{q}}).
\end{align*}

We now pass to the limit as $n\to\infty$ in \eqref{appxode1}. Consider $\psi(\cdot)$ be a $\mathrm{C}^1([0,T])$ function with  $\psi(T) = 0$. We multiply \eqref{appxode1} by $\psi(\cdot)$ and then do integration by parts, it leads to the equation 
\begin{align}\label{Fin-1}
	-&\int_0^T (\u_n(t), \psi'(t)\v)\d t + \int_0^T (\mathscr{F}(\u_n(t)),\psi(t)\v)\d t 
	+\sum_{i=1}^{d}\int_0^T \int_{\mathfrak{D}} \theta_{in}(t, \u_{n,i}(t)) \v_i \d x \psi(t) \d t \nonumber\\
	&= (\u_n(0),\v)\psi(0) + \int_0^T \langle \f(t), \psi(t)\v\rangle \d t, \text{ for all } \v \in \H_n.
\end{align}
Passing the limit $n\to\infty$ in \eqref{Fin-1} along with the weak convergences \eqref{wconvergence}, we find
\begin{align}\label{Fin-2}
	-&\int_0^T (\u(t), \psi'(t)\v)\d t 
	+ \int_0^T \langle \mathscr{F}(\u(t)), \psi(t)\v \rangle\d t +\sum_{i=1}^{d}\int_0^T \int_{\mathfrak{D}} \chi_i(t) \v_i \d x \psi(t) \d t\nonumber\\
	&= (\u^0,\v)\psi(0) + \int_0^T \langle \f(t), \psi(t)\v \rangle\d t, 
\end{align}
for all $\v =\w_1, \w_2, \ldots.$ By linearity, the above equality holds true for any finite linear combination of $\w_j$'s. Consequently, by a density argument,
% that is there exists a sequence $\{\w_j\}$ such that  $\w_j \in \V_n \subset \V$ and $\w_j \to \w \text{ in } \V$, one can also show that 
\eqref{Fin-2} is true for all $\w \in \V\cap\widetilde{\L}^{r+1}$.  Now, writing \eqref{Fin-2} with 
$\psi\in\mathscr{D}(0,T)$ (test function class), we deduce the following equality in the distributional sense: 
\begin{align}\label{appxode3}
	\frac{\d}{\d t}\big\langle \u(t),\w \big\rangle+\langle \mathscr{F}(\u(t)),\w\rangle +\sum_{i=1}^{d}\int_{\mathfrak{D}} \chi_i (t)\w_i \d x= \langle {\f}(t),\w\rangle,
\end{align}
for all $\w \in \V\cap\widetilde{\L}^{r+1}$ and for a.e. $t\in[0,T]$. Thus, we have shown that $\u$ satisfies the following weak formulation:
\begin{equation}\label{appxode4}
	\left\{
	\begin{aligned}
&\langle \u(t),\w \rangle+ \int_{0}^{t}\langle  \mathscr{F}(\u(s)),\w\rangle \d s +\sum_{i=1}^{d} \int_{0}^{t}\int_{\mathfrak{D}} \chi_i(s) \w_i \d x \d s\\&= \langle\u^0,\w\rangle+\int_{0}^{t}  \langle {\f}(s),\w\rangle \d s, \text{ for all } t\in[0,T] \ \text{ and all } \ \w \in \V\cap\widetilde{\L}^{r+1}.
	\end{aligned}
	\right.
\end{equation}
\vskip 0.2cm
\noindent
\textbf{Step V:} \emph{$\chi_i(t)\in \partial_{\mathpzc{C}} j_i(t,\u_i)$ a.e. in  $\mathfrak{D} \times (0,T).$} From \eqref{3p61}, we have
\begin{align*}
	\u_n(x,t) \to \u(x,t) \text{ a.e. in } \mathfrak{D} \times (0,T).
\end{align*}
Now, by applying Egorov's theorem \cite[Theorem 3.4.2]{CHL}, for every $\vartheta >0$, we can determine $\omega \subset \mathcal{H} = \mathfrak{D} \times (0,T)$ with $\mathcal{L}_{d+1}(\omega)<\vartheta$, where $\mathcal{L}_{d}(\omega)$  denotes the Lebesgue measure of dimension $d$, then we obtain
$$\u_n \to \u \  \text{ uniformly in } \ \mathcal{H}\setminus \omega, $$
with $\u \in \mathrm{L}^\infty(\mathcal{H}\setminus \omega ;\mathbb{R}^d)$. Let $\v \in \mathrm{L}^\infty(\mathcal{H}\setminus \omega ;\mathbb{R}^d)$. Due to Fatou’s lemma, for any $\varepsilon >0,$ there exists $\delta_\varepsilon >0$ and $N_\varepsilon$ such that (for $i=1,2,\ldots,d$)
\begin{align*}
\int_{\mathcal{H} \setminus \omega} 
\frac{j_i(t,\u_{n,i}(t)-\tau + \theta \v_i) - j_i(t,\u_{n,i}(t)-\tau)}{\theta} \d x \d t  \leq  \int_{\mathcal{H} \setminus \omega} j_i^0(t,\u_i(t);\v_i) \d x \d t + \varepsilon,
\end{align*}
provided $n > N_\varepsilon,~ |\tau| < \delta_{\varepsilon}$ and $0 < \theta < \delta_\varepsilon$. Since, $\int_{\mathbb{R}}\uprho_{\varepsilon_n}(s)\d s=1$, thus we have
\begin{align*}
&\int_{\mathbb{R}} \uprho_{\varepsilon_n}(\tau) 
\left( \int_{\mathcal{H} \setminus \omega} 
\frac{ j_i(t, \u_{n,i}(t) - \tau + \theta \v_i) - j_i(t,\u_{n,i}(t)- \tau)}{\theta} 
\d x \d t \right) \d \tau \nonumber\\&\leq 
\int_{\mathcal{H} \setminus \omega} j_i^0(t,\u_i(t); \v_i) \d x \d t + \varepsilon.
\end{align*}
Using the definition of convolution written in \eqref{Reg-1}, we obtain
\begin{align}\label{conver-5}
\int_{\mathcal{H} \setminus \omega} 
&\frac{ \uprho_{\varepsilon_n} * j_i(t,\u_{n,i}(t) + \theta \v_i) 
	- \uprho_{\varepsilon_n} * j_i(t,\u_{n,i}(t))}{\theta} 
\d x \d t \leq  \int_{\mathcal{H} \setminus \omega} j_i^0(t, \u_i(t); \v_i) \d x \d t + \varepsilon.
\end{align}
Taking the limit as $ \theta \to 0$ in above inequality \eqref{conver-5}, we get
\begin{align*}
\int_{\mathcal{H} \setminus \omega} 
[(\uprho_{\varepsilon_n} * j_i)^{\prime}(t,\u_{n,i}(t))] \cdot \v_i	\d x \d t \leq  \int_{\mathcal{H} \setminus \omega} j_i^0(t, \u_i(t); \v_i) \d x \d t + \varepsilon,
\end{align*}
where $'$ denotes the directional derivative. Equivalently by using the properties of mollifiers (see \cite{LCE}), \eqref{eqn-310} and \eqref{Reg-1}, we can write 
\begin{align*}
	\int_{\mathcal{H} \setminus \omega} 
	\theta_{in} (t,\u_{n,i}(t))\v_i	\d x \d t \leq  \int_{\mathcal{H} \setminus \omega} j_i^0(t, \u_i(t); \v_i) \d x \d t + \varepsilon.
\end{align*}
Now, taking the limit $n \to \infty$, it leads  to
$$ \int_{\mathcal{H} \setminus \omega} \chi_i(t) \v_i	\d x \d t \leq  \int_{\mathcal{H} \setminus \omega} j_i^0(t, \u_i(t); \v_i) \d x \d t + \varepsilon.$$
Since $\varepsilon > 0$ is chosen arbitrarily small, we conclude 
$$ \int_{\mathcal{H} \setminus \omega} \chi_i(t) \v_i	\d x \d t \leq  \int_{\mathcal{H} \setminus \omega} j_i^0(t, \u_i(t); \v_i) \d x \d t \ \text{ for all } \ \v_i \in \mathrm{L}^\infty(\mathcal{H}\setminus \omega ;\mathbb{R}).$$
By using the definition of Clarke subdifferential in \eqref{def-gra}, the last inequality implies that 
$$ \chi_i(t)\in \partial_{\mathpzc{C}} j_i(t,\u_i) \ \text{ for a.e.} \ (x,t)\in \mathcal{H} \setminus \omega, $$
with $\mathcal{L}_{d+1}(\omega)<\vartheta$. Now taking into account that $\vartheta$ was chosen arbitrarily, we finally obtain 
$$ \chi_i(t)\in \partial_{\mathpzc{C}} j_i(t,\u_i) \ \text{ a.e. in }\  \mathfrak{D} \times (0,T).$$
Thus, we have 
\begin{align}\label{3.48}
	\sum_{i=1}^{d} \int_{\mathfrak{D}}\chi_i(t)\v_i\d x \leq  \sum_{i=1}^{d} \int_{\mathfrak{D}} j_i^0(t, \u_i(t); \v_i) \d x \ \text{ for all } \ \v\in\V\cap\wi\L^{r+1},\text{ a.e. in }  (0,T).
\end{align}
Hence, by using the inequality \eqref{3.48} in \eqref{appxode3}, we get that the solution $\u$ satifies  Problem \ref{prob-inequality}, that is, 
\begin{align*}
	\frac{\d}{\d t}\big \langle \u(t),\w\big\rangle+\langle \mathscr{F}(\u(t)),\w\rangle +\sum_{i=1}^{d} \int_{\mathfrak{D}} j_i^0(t, \u_i(t); \w_i) \d x \geq \langle {\f}(t),\w\rangle \text{ for a.e. } \ t\in[0,T],
\end{align*}
for all $\w \in \V\cap\widetilde{\L}^{r+1}$.
\vskip 0.2cm
\noindent
\textbf{Step VI:} \emph{Initial data.} Let us multiply \eqref{appxode3} with $\psi$, integrate
with respect to $t$ and integrate by parts the first term to get
\begin{align}\label{Fin-22}
	-&\int_0^T (\u(t), \psi'(t)\v)\d t 
	+ \int_0^T \langle \mathscr{F}(\u(t)), \psi(t)\v \rangle\d t +\sum_{i=1}^{d}\int_0^T \int_{\mathfrak{D}} \chi_i(t) \v_i \d x \psi(t) \d t\nonumber\\
	&= (\u(0),\v)\psi(0) + \int_0^T \langle \f(t), \psi(t)\v \rangle\d t, 
\end{align}
for all $t\in[0,T]$. On comparing \eqref{Fin-22} with \eqref{Fin-2} and choosing $\psi$ in such a way that $\psi(0)=1$, we deduce the following:
\begin{align*}
	\big(\u(0)-\u^0,\v\big)=0 \ \text{ for all } \ \v\in\V\cap\wi\L^{r+1},
\end{align*}
which completes the proof. 
\end{proof}

\begin{proposition}\label{prop-energy}
Under the assumptions of Theorem \ref{thm 4.3}, the weak solution $\u$ of the Problem \ref{prob-inequality} satisfies the energy equality for $r\geq1$ in dimension two and $r\geq 3$ in dimension three. 
\end{proposition}

\begin{proof}
	%The energy equality is not immediate due to the {\color{purple}final convergence in \eqref{eqn-conv}}.
	%How we get final convergence on the function $\mathcal{F}(\u_n)$
	To prove the energy equality, we follow the approximation technique given in \cite{FHR}. In \cite{FHR}, the authors constructed approximations of $ \u(\cdot)$ in bounded domains such that the approximating sequence remains bounded and converges simultaneously in both Sobolev and Lebesgue spaces. Following this approach, we approximate $ \u(t)$ for each $t \in [0,T]$ by  employing the finite-dimensional space spanned by the first $ n$ eigen-functions of the Stokes operator (see \cite[Theorem 4.3]{FHR}), which is written as following:
	\begin{align}\label{approx-4}
		\u_n(t) := \mathcal{P}_{1/n}\u(t)= \sum_{\lambda_j < n^2} e^{-\lambda_j/n} \langle \u(t), \w_j \rangle \w_j .
	\end{align}
We point out that $\mathcal{P}_{1/n}$ is a self-adjoint operator but not a projection. 	Now, firstly note that
	\begin{equation}\label{4.46}
		\|\u_n\|_{\H}^2 = \|\mathcal{P}_{1/n}\u\|_{\H}^2 
		= \sum_{\lambda_j < n^2} e^{-2\lambda_j/n} |\langle \u, \w_j \rangle|^2 
		\leq \sum_{j=1}^{\infty} |\langle \u, \w_j \rangle|^2 
		= \|\u\|_{\H}^2 < +\infty, 
	\end{equation}
	for all $\u \in \H$. Moreover, it follows that
	\begin{align}\label{approx-1}
		\|(\mathrm{I}_d - \mathcal{P}_{1/n})\u\|_{\H}^2 
		&= \|\u\|_{\H}^2 - 2\langle \u, \mathcal{P}_{1/n}\u \rangle + \|\mathcal{P}_{1/n}\u\|_{\H}^2 
		\nonumber \\
		&= \sum_{\lambda_j < n^2} \big( 1 - e^{-\lambda_j/n} \big)^2 |\langle \u, \w_j \rangle|^2 
		+ \sum_{\lambda_j \geq n^2} |\langle \u, \w_j \rangle|^2,
	\end{align}
	for all $\u \in \H$. Since the series $\sum_{j=1}^\infty |\langle \u, \w_j \rangle|^2$ is convergent, it implies that the final term on the right-hand side of the equality \eqref{approx-1}, that is 
	\begin{align}\label{approx-2}
		\sum_{\lambda_j \geq n^2} |\langle \u, \w_j \rangle|^2 \to 0 \text{ as } n \to \infty.
	\end{align} 
	Also, the first term on the right-hand side of the equality can be bounded from above by
	\begin{align*}
		\sum_{j=1}^{\infty} \big( 1 - e^{-\lambda_j/n} \big)^2 |\langle \u, \w_j \rangle|^2 
		\leq 4 \sum_{j=1}^{\infty} |\langle \u, \w_j \rangle|^2 
		= 4 \|\u\|_{\H}^2 < +\infty.
	\end{align*}
	Thus, by using the Dominated Convergence Theorem, we obtain
	%we can interchange the limit and sum, and hence we obtain
	\begin{align}\label{approx-3}
		\lim_{n \to \infty} \sum_{j=1}^\infty \big(1 - e^{-\lambda_j/n}\big)^2 
		|\langle \u, \w_j \rangle|^2 
		= \sum_{j=1}^\infty \lim_{n \to \infty} \big(1 - e^{-\lambda_j/n}\big)^2 
		|\langle \u, \w_j\rangle|^2 = 0.
	\end{align}
	Hence, using \eqref{approx-2} and \eqref{approx-3} in \eqref{approx-1}, we get 
	\begin{align}\label{conv-H}
		\|(\mathrm{I}_d - \mathcal{P}_{1/n})\u\|_{\H} \to 0 \text{ as } n \to \infty. 
	\end{align}
	Furthermore, the authors in \cite{FHR} examine the properties of the approximation \eqref{approx-4}, which can be stated as follows:
	\begin{equation}\label{4.50}
		\left\{
		\begin{aligned}
			(1)~& \nabla\cdot\u_n(t)=0 \  \text{ and } \u_n|_{\partial \mathfrak{D}}=0 \text{ for all } t \in [0,T], \\
			(2)~& \lim\limits_{n\to\infty}\|\u_n(t)-\u(t)\|_{\mathbb{H}^1_0(\mathfrak{D})}=0 \text{ with } \|\u_n(t)\|_{\mathbb{H}^1} \leq C  \|\u(t)\|_{\mathbb{H}^1}
			 \text{ for all } t \in [0,T], \\
			(3)~& \lim\limits_{n\to\infty}\|\u_n(t)-\u(t)\|_{\L^p(\mathfrak{D})}=0 \text{ with } \|\u_n(t)\|_{\L^p} \leq C  \|\u(t)\|_{\L^p} \text{ for } p \in (1,\infty),  \\ &\text{ for a.e. } t \in [0,T].
		\end{aligned}
		\right.
	\end{equation}
	It should be noted that for $ d \in \{2,3\} $, 
	$\mathcal{D}(\A) \hookrightarrow \mathbb{H}^2(\mathfrak{D}) \hookrightarrow \L^p(\mathfrak{D})$, 
	for all $p \in (1,\infty)$ (cf. \cite{FHR}). 
	Since $\w_j$'s are the eigenfunctions of the Stokes operator $\A$, thus we obtain that 
	\begin{align*}
		\w_j \in \mathcal{D}(\A) \hookrightarrow \V \text{ and } \w_j \in \mathcal{D}(\A) \hookrightarrow \widetilde{\L}^{r+1}.	
	\end{align*}
	Since $\u \in \mathrm{L}^{r+1}(0,T;\widetilde{\L}^{r+1})$  and the fact obtained from (3), that is, 
	$\|\u_n(t)-\u(t)\|_{\widetilde{\L}^{r+1}} \to 0$, for a.e. $t \in [0,T]$, one can obtain the convergence \eqref{converg-1} (given below) by applying the dominated convergence theorem (with the dominating function $(1+C)\|\u(t)\|_{\widetilde{\L}^{r+1}}$), that is
	\begin{align}\label{converg-1}
		\|\u_n - \u\|_{\mathrm{L}^{r+1}(0,T; \widetilde{\L}^{r+1})} \to 0, \text{ as } n \to \infty.
	\end{align}
	Similarly, the fact that $\u \in\mathrm{L}^2(0,T;\V)$
	and from $(2)$ of \eqref{4.50}, we also have
	\begin{equation}\label{4.51}
		\|\u_n - \u\|_{\mathrm{L}^2(0,T;\V)} \to 0 \text{ as } n \to \infty.
	\end{equation}
	Now, from \eqref{appxode4}, for each $x$ and for all $t \in [0,\infty)$, we have
	\begin{equation*}
		\mathcal{P}_{1/n}\u(t,x) = 	\mathcal{P}_{1/n}\u^0(x) - \int_0^t 	\mathcal{P}_{1/n}\mathscr{F}(\u(s,x))\d s 
		+ \sum_{i=1}^{d} \int_0^t \mathcal{P}_{1/n} \chi_i(s,x) \d s +\int_0^t 	\mathcal{P}_{1/n}\f(s,x)\d s,
	\end{equation*}
where $\mathscr{F}(\cdot)$ is defined in \eqref{eqn-nem-f}.	Note that $	\mathcal{P}_{1/n}\u(\cdot, \cdot)$ is a smooth function in the second variable (since $	\mathcal{P}_{1/n}$ is a finite-dimensional projection). Now, for each $x \in \mathfrak{D}$, we have
	\begin{align}\label{ener-1}
		\frac{\d}{\d t}\mathcal{P}_{1/n}\u(t,x)=-\mathcal{P}_{1/n} \mathscr{F}(\u(t,x))+\sum_{i=1}^{d}\mathcal{P}_{1/n} \chi_i(t,x)+ \mathcal{P}_{1/n}\f(t,x),
	\end{align}
	and taking the inner product on both sides of \eqref{ener-1} with $\mathcal{P}_{1/n}\u(t,x)$, we obtain 
	\begin{align*}
		\frac{\d}{\d t}\mathcal{P}_{1/n}\u(t,x)\cdot\mathcal{P}_{1/n}\u(t,x)&=-\mathcal{P}_{1/n} \mathscr{F}(\u(t,x))\cdot\mathcal{P}_{1/n}\u(t,x)+\sum_{i=1}^{d}\mathcal{P}_{1/n} \chi_i(t,x)\cdot\mathcal{P}_{1/n}\u_i(t,x)\nonumber\\
		&+\mathcal{P}_{1/n}\f(t,x)\cdot\mathcal{P}_{1/n}\u(t,x).
	\end{align*}
	Integrating with respect to $t$ over the interval  $[0,t]$, we deduce 
	\begin{align}\label{ener-2}
		|\mathcal{P}_{1/n}&\u(t,x)|^2 
		= |\mathcal{P}_{1/n}\u^0(x)|^2 
		- 2 \int_0^t 	\mathcal{P}_{1/n}\mathscr{F}(\u(s,x)) \cdot 	\mathcal{P}_{1/n}\u(s,x)\d s \nonumber \\
		&+ 2 \sum_{i=1}^{d} \int_0^t \mathcal{P}_{1/n}\u_i(s,x) \cdot \mathcal{P}_{1/n}\chi_i(s,x) \d s +2\int_0^t \mathcal{P}_{1/n}\u(s,x) \cdot \mathcal{P}_{1/n}\f(s,x)\d s.
	\end{align}
	Now, we integrate \eqref{ener-2} over $\mathfrak{D}$ and obtain
	\begin{align}\label{4.55}
		&\|\mathcal{P}_{1/n} \u(t) \|_{\H}^{2} 
		= \| \mathcal{P}_{1/n} \u^0 \|_{\H}^{2}
		- 2 \int_{\mathfrak{D}} \int_{0}^{t} 
		\mathcal{P}_{1/n} \mathscr{F}(\u(s,x)) \cdot 
		\mathcal{P}_{1/n} \u(s,x) \d s \d x \nonumber \\
		&+ 2 \int_{\mathfrak{D}} \sum_{i=1}^{d}\int_{0}^{t} \mathcal{P}_{1/n} \u_i(s,x) \cdot 
		\mathcal{P}_{1/n} \chi_i(s,x) \d s \d x + 2\int_{\mathfrak{D}}\int_0^t \mathcal{P}_{1/n}\u(s,x) \cdot \mathcal{P}_{1/n}\f(s,x) \d s \d x.
	\end{align}
Now, let us consider $\mathscr{F}(\cdot)= \mathscr{F}^{1}(\cdot)+ \mathscr{F}^{2}(\cdot)$, where 
	\begin{align*}
		\mathscr{F}^{1}(\cdot):=\mu\A+\B(\cdot) \ \text{ and } \ \mathscr{F}^{2}(\cdot):=\alpha\wi{\C}(\cdot)+\beta\C(\cdot).
	\end{align*}
	Note that in $d=2$ for any $r\geq1$,
%	and in $d=3$ for $r=3$ 
by using Cauchy-Schwarz and H\"older's inequalities, we find
	\begin{align}\label{F1d2}
	\left|	\int_0^T\langle\mathscr{F}^{1}(\u(s)),\v(s)\rangle \d s\right|&=
	\left|	\int_0^T \langle\mu\A\u(s)+\B(\u(s)),\v(s)\rangle\d s\right|
		\nonumber\\&\leq
		\mu\int_0^T\|\u(s)\|_{\V}\|\v(s)\|_{\V}\d s+
		\int_0^T\|\u(s)\|_{\wi\L^4}^2\|\v(s)\|_{\V}\d s
		\nonumber\\&\leq
		\mu\|\u\|_{\mathrm{L}^2(0,T;\V)}\|\v\|_{\mathrm{L}^2(0,T;\V)}+
		\|\u\|_{\mathrm{L}^4(0,T;\wi\L^4)}^2\|\v\|_{\mathrm{L}^2(0,T;\V)},
	\end{align}
	for all $\v\in\mathrm{L}^{2}(0,T;\V)$. Moreover, in $d=3$ for $r\geq 3$, by using \eqref{bound-B},
%Cauchy Schwarz and H\"older's inequalities along with the interpolation inequality, 
we estimate
	\begin{align}\label{F1d3}
	\left|	\int_0^T \langle\mathscr{F}^{1}(\u(s)),\v(s)\rangle \d s\right|&=
\left|	\int_0^T \langle\mu\A\u(s)+\B(\u(s)),\v(s)\rangle \d s\right|
	\nonumber\\&\leq
	\mu\int_0^T\|\u(s)\|_{\V}\|\v(s)\|_{\V} \d s+
	\int_0^T\|\u(s)\|_{\widetilde{\L}^{r+1}}^{\frac{r+1}{r-1}}
	\|\u(s)\|_{\H}^{\frac{r-3}{r-1}}\|\v(s)\|_{\V} \d s\nonumber
	\\&\leq
	\mu\|\u\|_{\mathrm{L}^2(0,T;\V)}\|\v\|_{\mathrm{L}^2(0,T;\V)}+
	\|\u\|_{\mathrm{L}^{r+1}(0,T;\wi\L^{r+1})}^{\frac{r+1}{r-1}}
	\|\u\|_{\mathrm{L}^2(0,T;\H)}^{\frac{r-3}{r-1}}\|\v\|_{\mathrm{L}^2(0,T;\V)},
	\end{align}
	for all $\v\in\mathrm{L}^2(0,T;\V)$. 
%	Moreover, for $d=r=3$, from \eqref{F1dr3}, we immediately have 
%	\begin{align}\label{F1dr3e}
%			\int_0^T \langle\mathscr{F}^{1}(\u(s)),\v(s)\rangle \d s&=
%		\int_0^T \langle\mu\A\u(s)+\B(\u(s)),\v(s)\rangle \d s
%		\nonumber\\&\leq
%		\mu\int_0^T\|\u(s)\|_{\V}\|\v(s)\|_{\V} \d s+
%		\int_0^T\|\u(s)\|_{\widetilde{\L}^{4}}^{2}\|\v(s)\|_{\V} \d s
%		\nonumber\\&\leq
%			\mu\|\u\|_{\mathrm{L}^2(0,T;\V)}\|\v\|_{\mathrm{L}^2(0,T;\V)}+
%		\|\u\|_{\mathrm{L}^4(0,T;\wi\L^4)}^2\|\v\|_{\mathrm{L}^2(0,T;\V)},
%		\end{align}
%		for all $\v\in\mathrm{L}^{2}(0,T;\V)$. 
		Thus, from \eqref{F1d2}-\eqref{F1d3}, we conclude 
	\begin{align}\label{F1isf}
		\mathscr{F}^{1}(\u)\in\mathrm{L}^2(0,T;\V^{\prime}),
	\end{align}
	in $d=2$ for any $r\geq1$ and in $d=3$ for any $r\geq3$. Similarly, by the properties of the nonlinear operator $\C(\cdot)$ and $\wi\C(\cdot)$, one can prove 
	\begin{align}\label{F2isf}
		\mathscr{F}^{2}(\u)\in\mathrm{L}^{\frac{r+1}{r}}(0,T;\wi\L^{\frac{r+1}{r}}),
	\end{align}
in $d=2,3$ for any $r\geq1$.

Thus, by combining \eqref{F1isf} and \eqref{F2isf}, we conclude that $\mathscr{F}(\u)= \mathscr{F}^{1}(\u)+ \mathscr{F}^{2}(\u)$, where $	\mathscr{F}^{1}(\u)\in\mathrm{L}^2(0,T;\V^{\prime})$ and $	\mathscr{F}^{2}(\u)\in\mathrm{L}^{\frac{r+1}{r}}(0,T;\wi\L^{\frac{r+1}{r}})$ for $r\geq1$ in $d=2$ and $r\geq3$ in $d=3$. 

Now, by using the Cauchy-Schwarz and H\"older inequalities, and properties $(2)$ and $(3)$ from \eqref{4.50}, we calculate
	\begin{align}\label{4.56}
		&\left|\int_{0}^{t} \int_{\mathfrak{D}} 
		\mathcal{P}_{1/n} \mathscr{F}(\u(s,x))\cdot \mathcal{P}_{1/n} \u(s,x)\d x\d s \right|\nonumber\\
		&= \left|\int_{0}^{t} 
		\langle\mathscr{F}(\u(s)), \mathcal{P}_{1/n} \mathcal{P}_{1/n} \u(s)\rangle\d s
		\right|
		\nonumber\\
		&\le \int_{0}^{t} 
		\|\mathscr{F}^{1}(\u(s))\|_{\V^{\prime}} 
		\|\mathcal{P}_{1/n} \mathcal{P}_{1/n}\u(s)\|_{\V} \d s
		+ \int_{0}^{t} 
		\|\mathscr{F}^{2}(\u(s))\|_{\widetilde\L^\frac{r+1}{r}}
		\|\mathcal{P}_{1/n} \mathcal{P}_{1/n}\u(s)\|_{\widetilde\L^{r+1}} \d s \nonumber\\
		&\le C \int_{0}^{t} 
		\|\mathscr{F}^{1}(\u(s))\|_{\V^{\prime}} 
		\|\mathcal{P}_{1/n}\u(s)\|_{\V} \d s
		+ C \int_{0}^{t} 
		\|\mathscr{F}^{2}(\u(s))\|_{\widetilde\L^\frac{r+1}{r}}
		\|\mathcal{P}_{1/n}\u(s)\|_{\widetilde\L^{r+1}} \d s \nonumber\\
		&\le C\int_{0}^{t} 
		\|\mathscr{F}^{1}(\u(s))\|_{\V^{\prime}} 
		\|\u(s)\|_{\V} \d s
		+ C\int_{0}^{t} 
		\|\mathscr{F}^{2}(\u(s))\|_{\widetilde\L^\frac{r+1}{r}}
		\|\u(s)\|_{\widetilde\L^{r+1}} \d s \nonumber\\
		&\le C 
		\left( \int_{0}^{t} 
		\|\mathscr{F}^{1}(\u(s))\|_{\V^{\prime}}^{2} \d s \right)^{1/2}
		\left( \int_{0}^{t} 
		\|\u(s)\|_{\V}^{2} \d s \right)^{1/2} \nonumber\\
		&\quad + C
		\left( \int_{0}^{t} 
		\|\mathscr{F}^{2}(\u(s))\|_{\widetilde\L^\frac{r+1}{r}}^{\frac{r+1}{r}} \d s \right)^{\frac{r}{r+1}}
		\left( \int_{0}^{t} 
		\|\u(s)\|_{\widetilde\L^{r+1}}^{r+1} \d s \right)^{\frac{1}{r+1}}.
	\end{align}
From \eqref{F1isf}-\eqref{F2isf} along with the fact that $\u \in \mathrm{L}^{2}(0,T;\V) \cap \mathrm{L}^{r+1}(0,T;\widetilde\L^{r+1}),$ the inequality \eqref{4.56} yields
	\begin{align}\label{4.58}
	\left|\int_{0}^{t} \int_{\mathfrak{D}}
	\mathcal{P}_{1/n} \mathscr{F}(\u(s,x))\cdot \mathcal{P}_{1/n} \u(s,x)\d x\d s \right|< +\infty.
	\end{align}
	Similarily, we compute
	\begin{align*}
		\left|\int_0^t \int_{\mathfrak{D}}\mathcal{P}_{1/n}\u(s,x) \cdot \mathcal{P}_{1/n}\f(s,x) \d x \d s \right|
		&= \left|\int_{0}^{t} \langle \mathcal{P}_{1/n} \mathcal{P}_{1/n} \u(s),\f(s) \rangle \d s \right|\nonumber\\
		%&\leq \int_{0}^{t} \int_{\mathfrak{D}}
		%|\f(s,x)| | \mathcal{P}_{1/n}\mathcal{P}_{1/n}\u(s,x) | \d x \d s \nonumber\\
		&\le \int_{0}^{t} 
		\|\f(s)\|_{\V^{\prime}} 
		\|\mathcal{P}_{1/n} \mathcal{P}_{1/n}\u(s)\|_{\V} \d s \nonumber\\
		&\le C \int_{0}^{t} 
		\|\f(s)\|_{\V^{\prime}} 
		\|\mathcal{P}_{1/n}\u(s)\|_{\V} \d s \nonumber\\
		&\le C  \left( \int_{0}^{t}  \|\f(s)\|_{\V^{\prime}}^{2} \d s \right)^{1/2} \left( \int_{0}^{t} 
		\|\u(s)\|_{\V}^{2} \d s \right)^{1/2}.
	\end{align*}
	Since $\f \in \mathrm{L}^2(0,T;\V^{\prime}) \text{ and } \u \in \mathrm{L}^{2}(0,T;\V)$, we get
	\begin{align}\label{4.60}
	\left|\int_0^t \int_{\mathfrak{D}}\mathcal{P}_{1/n}\u(s,x) \cdot \mathcal{P}_{1/n}\f(s,x) \d x \d s \right|< +\infty .
	\end{align} 
	Now, we calculate 
	\begin{align*}
		\left|\int_{0}^{t}\int_{\mathfrak{D}} \mathcal{P}_{1/n} \u_i(s,x) \cdot \mathcal{P}_{1/n} \chi_i(s,x) \d x \d s\right|  
		&= \left|\int_{0}^{t} ( \mathcal{P}_{1/n} \mathcal{P}_{1/n} \u_i(s),\chi_i(s) )\d s
		\right| \nonumber\\
		&\le \int_{0}^{t} \|\mathcal{P}_{1/n} \mathcal{P}_{1/n}\u_i(s)\|_{\H} \|\chi_i(s)\|_{\H} \d s \nonumber\\
		&\le  \int_{0}^{t} \|\u_i(s)\|_{\H} \|\chi_i(s)\|_{\H} \d s \nonumber\\
		&\le  \|\u_i\|_{\mathrm{L}^2(0,T;\H)} \|\chi_i\|_{\mathrm{L}^2(0,T;\H)}.
	\end{align*}
	Since, using the fact that $\mathrm{L}^{\infty}(0,T;\H) \hookrightarrow \mathrm{L}^2(0,T;\H)$, \eqref{4.46}, and $\chi_i \in \mathrm{L}^2(0,T;\H)$ in the above inequality, we obtain 
	\begin{align}\label{4.67}
	\left|\int_{0}^{t}\int_{\mathfrak{D}} \mathcal{P}_{1/n} \u_i(s,x) \cdot \mathcal{P}_{1/n} \chi_i(s,x) \d x \d s\right|  < +\infty.
	\end{align}
	Hence, by using \eqref{4.58}, \eqref{4.60}, \eqref{4.67} and an application of Fubini's theorem (see, for instance, \cite[Theorem 14.1]{RA}), we deduce from \eqref{4.55} that
	\begin{align}\label{4.62}
	&	\|\mathcal{P}_{1/n}\u(t)\|_{\H}^{2}
		= \|\mathcal{P}_{1/n}\u^0\|_{\H}^{2}
		- 2 \int_{0}^{t} \int_{\mathfrak{D}} \mathcal{P}_{1/n}\mathscr{F}(\u(s,x)) \cdot \mathcal{P}_{1/n}\u(s,x) \d x \d s \nonumber \\
		&\quad+2\int_{0}^{t}\sum_{i=1}^{d} \int_{\mathfrak{D}} \mathcal{P}_{1/n} \u_i(s,x) \cdot \mathcal{P}_{1/n} \chi_i(s,x) \d x \d s +2\int_0^t \int_{\mathfrak{D}}\mathcal{P}_{1/n}\u(s,x) \cdot \mathcal{P}_{1/n}\f(s,x) \d x \d s \nonumber\\
		& = \|\mathcal{P}_{1/n}\u^0\|_{\H}^{2}
		- 2 \int_{0}^{t} \langle \mathscr{F}(\u(s)), \mathcal{P}_{1/n} \mathcal{P}_{1/n}\u(s) \rangle \d s \nonumber \\
		&\quad+2\int_{0}^{t}\sum_{i=1}^{d} (\chi_i(s),\mathcal{P}_{1/n} \mathcal{P}_{1/n} \u_i(s)) \d s +2\int_0^t \langle \f(s),\mathcal{P}_{1/n}\mathcal{P}_{1/n}\u(s) \rangle \d s,
	\end{align}
	for all $t\in[0,T]$. Next we pass the limit as $n\to\infty$ in \eqref{4.62}. We first show the convergence
	 $$\left| \int_0^t \langle \mu\A \u(s)+\B(\u(s))+\alpha\wi{\C}(\u(s))+\beta \C(\u(s)), \mathcal{P}_{1/n} \mathcal{P}_{1/n} \u(s)- \u(s) \rangle \d s \right|\to 0, \text{ as } n \to \infty.$$ 
	 To do this, we first compute
\begin{align}\label{4.63}
	& \left| \int_0^t \langle \A \u(s), \mathcal{P}_{1/n} \mathcal{P}_{1/n} \u(s) -\u(s) \rangle \d s \right| \nonumber\\
		&=  \left| \int_0^t 
		\langle \A \u(s), \mathcal{P}_{1/n}(\mathcal{P}_{1/n}\u(s) - \u(s)) + \mathcal{P}_{1/n}\u(s) - \u(s) \rangle \d s \right| \nonumber\\
		& \leq \|\u\|_{\mathrm{L}^2(0,T;\V)}
		\left( \|\mathcal{P}_{1/n}(\mathcal{P}_{1/n}\u - \u)\|_{\mathrm{L}^2(0,T;\V)} 
		+ \|\mathcal{P}_{1/n}\u - \u\|_{\mathrm{L}^2(0,T;\V)} \right) \nonumber\\
		&\leq \| \u\|_{\mathrm{L}^2(0,T;\V)}	(C\|\mathcal{P}_{1/n}\u - \u \|_{\mathrm{L}^2(0,T;\V)} +  \|\mathcal{P}_{1/n}\u - \u \|_{\mathrm{L}^2(0,T;\V)}) \nonumber\\
				&  \to 0, \text{ as } n \to \infty,
	\end{align}
where we have used \eqref{4.51}. Now, by using \eqref{converg-1}, we estimate
\begin{align}\label{4.6}
	& \left| \int_0^t \langle \C (\u(s)), \mathcal{P}_{1/n} \mathcal{P}_{1/n} \u(s) -\u(s) \rangle \d s \right| \nonumber\\
	&=  \left| \int_0^t 
	\langle |\u(s)|^{r-1}\u(s), \mathcal{P}_{1/n}(\mathcal{P}_{1/n}\u(s) - \u(s)) + \mathcal{P}_{1/n}\u(s) - \u(s) \rangle \d s \right| \nonumber\\
	& \leq \|\u\|_{\mathrm{L}^{\frac{r+1}{r}}(0,T;\widetilde\L^{\frac{r+1}{r}})}
	\left( \|\mathcal{P}_{1/n}(\mathcal{P}_{1/n}\u - \u)\|_{\mathrm{L}^{r+1}(0,T;\widetilde\L^{r+1})} 
	+ \|\mathcal{P}_{1/n}\u - \u\|_{\mathrm{L}^{r+1}(0,T;\widetilde\L^{r+1})} \right) \nonumber\\
	&\leq	(C+1)\|\u\|_{\mathrm{L}^{\frac{r+1}{r}}(0,T;\widetilde\L^{\frac{r+1}{r}})} \|\mathcal{P}_{1/n}\u - \u\|_{\mathrm{L}^{r+1}(0,T;\widetilde\L^{r+1})} \nonumber\\
	&  \to 0, \text{ as } n \to \infty,
\end{align}
Similar to the calculations as in  \eqref{4.6}, we estimate 
\begin{align*}
	& \left| \int_0^t \langle \wi{\C} (\u(s)), \mathcal{P}_{1/n} \mathcal{P}_{1/n} \u(s) -\u(s) \rangle \d s \right| \nonumber\\
	&=  \left| \int_0^t 
	\langle |\u(s)|^{q-1}\u(s), \mathcal{P}_{1/n}(\mathcal{P}_{1/n}\u(s) - \u(s)) + \mathcal{P}_{1/n}\u(s) - \u(s) \rangle \d s \right| \nonumber\\
	&\leq	(C+1)\|\u\|_{\mathrm{L}^{\frac{r+1}{r}}(0,T;\widetilde\L^{\frac{r+1}{r}})} \|\mathcal{P}_{1/n}\u - \u\|_{\mathrm{L}^{q+1}(0,T;\widetilde\L^{q+1})} \nonumber\\
	&\leq	(C+1)\|\u\|_{\mathrm{L}^{\frac{r+1}{r}}(0,T;\widetilde\L^{\frac{r+1}{r}})} \|\mathcal{P}_{1/n}\u - \u\|_{\mathrm{L}^{r+1}(0,T;\widetilde\L^{r+1})}\nonumber\\
	&  \to 0 \  \text{ as } \ n \to \infty,
\end{align*}
where we have used the fact that $\mathrm{L}^{r+1}(0,T;\widetilde\L^{r+1})\hookrightarrow \mathrm{L}^{q+1}(0,T;\widetilde\L^{q+1})$ for $q < r$ and the convergence \eqref{converg-1}. Now, we discuss the convergence for the bilinear term for different values of $r$ written as follows:
\vskip 2mm
\noindent
\textbf{Case I:} \emph{For $r\geq1$ in $d=2$}. By using H\"older's, Ladyzhenskaya's inequalities and \eqref{4.51}, we compute
\begin{align*}
& \left| \int_0^t \langle \B(\u(s)), \mathcal{P}_{1/n} \mathcal{P}_{1/n} \u(s) -\u(s) \rangle \d s \right| \nonumber\\
&=  \left| \int_0^t 
\langle \B(\u(s)), \mathcal{P}_{1/n}(\mathcal{P}_{1/n}\u(s) - \u(s)) + \mathcal{P}_{1/n}\u(s) - \u(s) \rangle \d s \right| \nonumber\\
&=  \left| \int_0^t 
\langle \B(\u(s)), \mathcal{P}_{1/n}(\mathcal{P}_{1/n}\u(s) - \u(s))\rangle \d s \right| +\left| \int_0^t 
\langle \B(\u(s)), \mathcal{P}_{1/n}\u(s) - \u(s) \rangle \d s \right| \nonumber\\
&\leq \int_{0}^{t}\|\u(s)\|_{\wi\L^4}^2\|\mathcal{P}_{1/n}(\mathcal{P}_{1/n}\u(s) - \u(s))\|_{\V} \d s + \int_{0}^{t}\|\u(s)\|_{\wi\L^4}^2\|\mathcal{P}_{1/n}\u(s) - \u(s)\|_{\V} \d s \nonumber\\
&\leq \int_{0}^{t} \|\u(s)\|_{\H}\|\u(s)\|_{\V} \|\mathcal{P}_{1/n}(\mathcal{P}_{1/n}\u(s) - \u(s))\|_{\V} \d s + \int_{0}^{t} \|\u(s)\|_{\H}\|\u(s)\|_{\V}\|\mathcal{P}_{1/n}\u(s) - \u(s)\|_{\V} \d s\nonumber\\
&\leq (C+1)\|\u\|_{\mathrm{L}^{\infty}(0,T;\H)}\|\u\|_{\mathrm{L}^{2}(0,T;\V)} \|\mathcal{P}_{1/n}\u - \u\|_{\mathrm{L}^{2}(0,T;\V)} \nonumber\\
	&  \to 0, \text{ as } n \to \infty.
\end{align*}
\textbf{Case II:} \emph{For $r\geq 3$ in $d=3$}. By using H\"older's and interpolation inequalities, and \eqref{converg-1}, we find
\begin{align*}
& \left| \int_0^t \langle \B(\u(s)), \mathcal{P}_{1/n} \mathcal{P}_{1/n} \u(s) -\u(s) \rangle \d s \right| \nonumber\\
&=  \left| \int_0^t 
\langle \B(\u(s)), \mathcal{P}_{1/n}(\mathcal{P}_{1/n}\u(s) - \u(s))\rangle \d s \right| +\left| \int_0^t 
\langle \B(\u(s)), \mathcal{P}_{1/n}\u(s) - \u(s) \rangle \d s \right| \nonumber\\
&\leq (C+1)\int_0^t\|\u(s)\|_{\wi\L^{\frac{2(r+1)}{r-1}}}\|\u(s)\|_{\V}\|\mathcal{P}_{1/n}\u(s) - \u(s)\|_{\wi\L^{r+1}}\d s  \nonumber\\
&\leq (C+1) \int_0^t \|\u(s)\|_{\H}^{\frac{r-3}{r-1}}\|\u(s)\|_{\wi\L^{r+1}}^{\frac{2}{r-1}}\|\u(s)\|_{\V}
\|\mathcal{P}_{1/n}\u(s) - \u(s)\|_{\wi\L^{r+1}}\d s 
\nonumber\\&\leq 
(C+1)T^{\frac{r-3}{2(r-1)}}\|\u\|_{\mathrm{L}^{\infty}(0,T;\H)}^{\frac{r-3}{r-1}}
\|\u\|_{\mathrm{L}^{r+1}(0,T;\wi\L^{r+1})} \|\u\|_{\mathrm{L}^{2}(0,T;\V)}\|\mathcal{P}_{1/n}\u - \u\|_{\mathrm{L}^{r+1}(0,T;\wi\L^{r+1})} 
\nonumber\\& \to 0 \ \text{ as } \ n \to \infty.
\end{align*}
%Also, for $r=3$, we have
%\begin{align*}
%& \left| \int_0^t \langle \B(\u(s)), \mathcal{P}_{1/n} \mathcal{P}_{1/n} \u(s) -\u(s) \rangle \d s \right| \nonumber\\
%&=  \left| \int_0^t 
%\langle \B(\u(s)), \mathcal{P}_{1/n}(\mathcal{P}_{1/n}\u(s) - \u(s))\rangle \d s \right| +\left| \int_0^t 
%\langle \B(\u(s)), \mathcal{P}_{1/n}\u(s) - \u(s) \rangle \d s \right| \nonumber\\
%&\leq (C+1)\int_0^t\|\u(s)\|_{\wi\L^{4}}\|\u(s)\|_{\V}\|\mathcal{P}_{1/n}\u(s) - \u(s)\|_{\wi\L^{4}}\d s 
%\nonumber\\&\leq 
%	(C+1)
%	\|\u\|_{\mathrm{L}^{4}(0,T;\wi\L^{4})} \|\u\|_{\mathrm{L}^{2}(0,T;\V)}\|\mathcal{P}_{1/n}\u - \u\|_{\mathrm{L}^{4}(0,T;\wi\L^{4})} 
%\nonumber\\&\to 0, \text{ as } n \to \infty.
%\end{align*}
Moreover, by using \eqref{4.51}, we get
	\begin{align}\label{4.64}
		&\left| \int_{0}^{t} \langle \f(s), \mathcal{P}_{1/n} \mathcal{P}_{1/n} \u(s) - \u(s) \rangle \d s \right| \nonumber\\
		&= \left| \int_0^t 
		\langle \f(s) , \mathcal{P}_{1/n}(\mathcal{P}_{1/n}\u(s) - \u(s)) + \mathcal{P}_{1/n}\u(s) - \u(s) \rangle \d s \right| \nonumber\\
		& \leq \|\f\|_{\mathrm{L}^2(0,T;\V^{\prime})}
		\left( \|\mathcal{P}_{1/n}(\mathcal{P}_{1/n}\u - \u)\|_{\mathrm{L}^2(0,T;\V)} 
		+ \|\mathcal{P}_{1/n}\u - \u\|_{\mathrm{L}^2(0,T;\V)} \right) \nonumber\\
		&\leq \| \f\|_{\mathrm{L}^2(0,T;\V^{\prime})}	(C\|\mathcal{P}_{1/n}\u - \u \|_{\mathrm{L}^2(0,T;\V)} +  \|\mathcal{P}_{1/n}\u - \u \|_{\mathrm{L}^2(0,T;\V)}) \nonumber\\
		& \to 0\  \text{ as } \ n \to \infty.
	\end{align}
	Similarily, one can show that 
	\begin{align}\label{4.66}
		\left|\int_{0}^{t} (\chi_i(s),\mathcal{P}_{1/n} \mathcal{P}_{1/n} \u_i(s) - \u_i(s)) \d s \right| \to 0\  \text{ as } \ n \to \infty.
	\end{align}
	By passing the limit $n \to \infty$ in \eqref{4.62} and using \eqref{conv-H}, \eqref{4.63}, \eqref{4.64} and \eqref{4.66}, we immediately obtain the following energy equality:
	\begin{align*}
		\|\u(t)\|_{\H}^{2}
		&= \|\u^0\|_{\H}^{2}
		- 2 \int_{0}^{t} \langle \mathscr{F}(\u(s)), \u(s) \rangle \d s +2\sum_{i=1}^{d} \int_{0}^{t} \int_{\mathfrak{D}}\chi_i(s)\u_i(s)\d x \d s 
		\nonumber\\&\quad+2\int_0^t \langle \f(s),\u(s) \rangle \d s,
	\end{align*}
	for all $t \in [0,T]$. 
\end{proof}

\begin{remark}
Note that for $r\in[1,3)$ in $d=3$, the inequality \eqref{F1d3} no longer holds. Consequently, it follows from \eqref{F1isf} that $\mathscr{F}^1(\u)\notin\mathrm{L}^2(0,T;\V^{\prime})$ and therefore the relation \eqref{4.56} fails. As a result, the energy equality cannot be justified for $r\in[1,3)$  in three dimensions.
\end{remark}

We now turn to the question of the uniqueness of solutions. Unlike existing results for evolution boundary type hemivariational NSE problems (for instance see \cite{SMAON}), which only address existence, the presence of an absorption term $\beta|\u|^{r-1}\u$ in our model \eqref{eqn-dom} enables us to establish uniqueness. In particular, for $r>3,$ the additional dissipative effect of the absorption term plays a crucial role in controlling the nonlinearities and leads to the uniqueness of weak solutions.
\begin{theorem}\label{lemma4.6}
Let $\theta_i : \mathfrak{D} \times (0,T) \times \mathbb{R} \to \mathbb{R}$ satisfy  the Hypothesis \ref{hyptheta} and assume that the following condition holds:
\begin{align}\label{theta-Cond}
\essinf\limits_{\xi_1 \ne \xi_2} 
\frac{\theta_i(x,t,\xi_1) - \theta_i(x,t,\xi_2)}{\xi_1 - \xi_2} 
\geq -\mathcal{K} \ \text{ for all } \ (x,t) \in \mathfrak{D} \times (0,T)
\end{align}
with $\mathcal{K} > 0$. Then, for $r\geq1$ in dimension $d=2$, for $r>3$ in dimension $d=3$, and in the case $d=r=3$ provided that $2\beta\mu >1$, the Problem \ref{prob-inequality} admits a unique solution.
\end{theorem}

\begin{proof}
Let $(\u^1,\boldsymbol{\chi}^1)$ and $(\u^2,\boldsymbol{\chi}^2),$ where $\boldsymbol{\chi}^k=(\chi^k_1,\ldots,\chi^k_d)$, for $k=1,2$, be two weak solutions of  \eqref{appxode3}, with $\u^1,\u^2\in\mathcal{W}$ and 
$\boldsymbol{\chi}^k\in\partial_{\mathpzc{C}} j(\cdot,\u^k)$ for $k=1,2$ satisfying \eqref{eqn-clarke}, corresponding to the external forcing $\f\in\mathrm{L}^2(0,T;\V^{\prime})$ and the initial data $\u^0\in\H$, respectively. Then, $(\w,\boldsymbol{\chi})=(\u^1-\u^2,\boldsymbol{\chi}^1-\boldsymbol{\chi}^2)$ satisfies the following:
	\begin{equation}\label{eqn-abs-11}
	\left\{
	\begin{aligned}
		&\frac{\d}{\d t}\langle\w,\v\rangle+\langle\mu\A\w,\v\rangle
		+\langle\B(\u^1)-\B(\u^2),\v\rangle+
		\alpha\langle\wi{\C}(\u^1)-\wi{\C}(\u^2),\v\rangle
		\nonumber\\&\quad+\beta\langle\C(\u^1)-\C(\u^2),\v\rangle+
		(\boldsymbol{\chi},\v)=0, \ \text{ for a.e. } \ t\in[0,T],
	\end{aligned}
	\right.
\end{equation}
 for all $\v\in\V\cap\L^{r+1}$. Then, from Proposition \ref{prop-energy}, $\w$ satisfies the following energy equality:
\begin{align}\label{eqn-unique-non}
	&\|\w(t)\|_{\H}^2+2\mu\int_0^t\|\w(s)\|_{\V}^2\d s+
	2\beta\int_0^t\langle\C(\u^1(s))-\C(\u^2(s)),\w(s)\rangle\d s
	 \nonumber\\&=
	\|\w(0)\|_{\H}^2-2\int_0^t\langle\B(\u^1(s))-\B(\u^2(s)),\w(s)\rangle\d s -2\alpha \int_0^t\langle\wi{\C}(\u^1(s))-\wi{\C}(\u^2(s)),\w(s)\rangle\d s \nonumber\\&\quad
	-2\int_0^t({\boldsymbol{\chi}}(s),\w(s))\d s,
\end{align}
for all $t\in[0,T]$. Now, from the assumption \eqref{theta-Cond}, we get 
\begin{align}\label{inequality-1}
\inf_{\xi_1 > \xi_2} 
			\frac{\underline{\theta}(x, t,\xi_1) - \overline{\theta}(x, t, \xi_2)}{\xi_1 - \xi_2}
			\geq -\mathcal{K} 	\quad \text{for all } (x, t) \in \mathfrak{D} \times (0,T).
\end{align}
Let $\Omega_1(s) = \{x \in \mathfrak{D} : \u^1(x, s) > \u^2(x, s)\}$ 
and $\Omega_2(s) = \{x \in \mathfrak{D} : \u^2(x, s) > \u^1(x, s)\}$ for all $s \in[0,T]$. 
Using \eqref{Equality}, the definition of $\widehat{\theta}$ in \eqref{Def-theta} and \eqref{inequality-1}, we have
\begin{align}\label{4.81}
	&(\boldsymbol{\chi}^1(s) - \boldsymbol{\chi}^2(s), \u^1(s) - \u^2(s))\nonumber\\
	&= \int_{\Omega_1(s)} (\boldsymbol{\chi}^1(x, s) - \boldsymbol{\chi}^2(x, s))\cdot(\u^1(x, s) - \u^2(x, s)) \d x \nonumber\\
	&\quad + \int_{\Omega_2(s)} (\boldsymbol{\chi}^1(x, s) - \boldsymbol{\chi}^2(x, s))\cdot(\u^1(x, s) - \u^2(x, s)) \d x \nonumber\\
	&\geq \int_{\Omega_1(s)} 
	\Big[ \underline{\theta}(x, s, \u^1(x, s)) - 
	\overline{\theta}(x, s, \u^2(x, s)) \Big] \cdot
	(\u^1(x, s) - \u^2(x, s)) \d x \nonumber\\
	&\quad + \int_{\Omega_2(s)} 
	\Big[ \underline{\theta}(x, s,  \u^2(x, s)) - 
	\overline{\theta}(x, s, \u^1(x, s)) \Big] \cdot
	(\u^2(x, s) - \u^1(x, s)) \d x \nonumber\\
	&\ge -\mathcal{K}  \int_{\Omega_1(s)} |\u^1(x, s) - \u^2(x, s)|^2 \d x
	- \mathcal{K}  \int_{\Omega_2(s)} |\u^2(x, s) - \u^1(x, s)|^2 \d x \nonumber\\
	&= -\mathcal{K}  \| \u^1(s) - \u^2(s) \|_{\H}^2.
\end{align}
Now, in view of Lemma \ref{Cmono1}, and by an application of Taylor's formula and  H\"older's inequality, we have  the following estimates:
\begin{align}
   \langle\C(\u^1) -\C(\u^2),\w\rangle
	&\geq\frac{\beta}{2} \||\u^1|^{\frac{r-1}{2}}\w\|^2_{\H}
	+ \frac{\beta}{2}\||\u^2|^{\frac{r-1}{2}}\w\|^2_{\H},\label{4.76}\\
	|\alpha| |\langle\wi{\C}(\u^1)-\wi{\C}(\u^2),\w\rangle|
	&\leq \frac{\beta}{4} \||\u^1|^{\frac{r-1}{2}}\w\|^2_{\H}
	+ \frac{\beta}{4} \||\u^2|^{\frac{r-1}{2}}\w\|^2_{\H}
	+(\varrho_1 + \varrho_2) \|\w\|^2_{\H},\label{4.766}
\end{align}
where $\varrho_1= \left( \frac{r - q}{r - 1}\right)
\left( \frac{2^{q+1} q |\alpha|(q - 1)}{\beta (r - 1)} \right)^{\frac{q - 1}{r - q}} \text{ and } \varrho_2= \left( \frac{r - q}{r - 1}\right)
\left( \frac{2^{q+1} q |\alpha|(q - 1)}{\beta (r - 1)} \right)^{\frac{q - 1}{r - q}} $.
On substituting \eqref{4.81}-\eqref{4.766} into \eqref{eqn-unique-non}, we obtain
\begin{align}\label{systunq}
	&\|\w(t)\|_{\H}^2+2\mu\int_0^t\|\w(s)\|_{\V}^2\d s+
	\frac{\beta}{2}\int_0^t\||\u^1(s)|^{\frac{r-1}{2}}\w(s)\|^2_{\H}\d s
	+ \frac{\beta}{2}\int_0^t\||\u^2(s)|^{\frac{r-1}{2}}\w(s)\|^2_{\H}\d s
	\nonumber\\&\leq
	\|\w(0)\|_{\H}^2+2(\varrho_1 +\varrho_2+\mathcal{K})\int_0^t\|\w(s)\|_{\H}^2-
	2\int_0^t\langle\B(\u^1(s))-\B(\u^2(s)),\w(s)\rangle\d s,
\end{align}
for all $t\in[0,T]$. 
Now, we discuss the cases for different values of $r$ written as follows:
\vskip 2mm
\noindent
\textbf{Case I:} \emph{For $r\geq1$ in $d=2$}. By using H\"older's, Ladyzhenskaya's and Young's inequalities, we compute
\begin{align}\label{b1b2df}
|\langle \B(\u^1)-\B(\u^2),\w\rangle|&=|\langle \B(\w, \u^2), \w \rangle|
=|b(\w, \u^2, \w)|
\nonumber\\&\leq\|\w\|_{\wi\L^4}\|\w\|_{\V}\|\u^2\|_{\wi\L^4}
\leq\|\w\|_{\H}^{\frac12}\|\w\|_{\V}^{\frac32}\|\u^2\|_{\wi\L^4}
\nonumber\\&\leq \frac{\mu}{2}\|\w\|_{\V}^2+\frac{27}{32\mu^3}\|\w\|_{\H}^2\|\u^2\|_{\wi\L^4}^4.
\end{align}
Then \eqref{systunq} along with \eqref{b1b2df} yields
\begin{align}\label{systunq1}
	&\|\w(t)\|_{\H}^2+\mu\int_0^t\|\w(s)\|_{\V}^2\d s+
	\frac{\beta}{2}\int_0^t\||\u^1(s)|^{\frac{r-1}{2}}\w(s)\|^2_{\H}\d s
	+ \frac{\beta}{2}\int_0^t\||\u^2(s)|^{\frac{r-1}{2}}\w(s)\|^2_{\H}\d s
	\nonumber\\&\leq
	\|\w(0)\|_{\H}^2+2(\varrho_1+\varrho_2+\mathcal{K}) \int_0^t\|\w(s)\|_{\H}^2+\frac{27}{16\mu^3}
	\int_0^t\|\w(s)\|_{\H}^2\|\u^2(s)\|_{\wi\L^4}^4\d s,
\end{align}
for all $t\in[0,T]$. On employing Gr\"onwall's Lemma into \eqref{systunq1}, one obtains
\begin{align}\label{systunq2}
	&\|\w(t)\|_{\H}^2+\mu\int_0^t\|\w(s)\|_{\V}^2\d s+
\frac{\beta}{2}\int_0^t\||\u^1(s)|^{\frac{r-1}{2}}\w(s)\|^2_{\H}\d s
+ \frac{\beta}{2}\int_0^t\||\u^2(s)|^{\frac{r-1}{2}}\w(s)\|^2_{\H}\d s
\nonumber\\&\leq
\|\w(0)\|_{\H}^2 e^{2(\varrho_1+\varrho_2+\mathcal{K})t}
\exp\bigg(\frac{27}{16\mu^3}\int_0^t \|\u^2(s)\|_{\wi\L^4}^4\d s\bigg),
\end{align}
for all $t\in[0,T]$. Since, $\w(0)=\boldsymbol{0}$ and $\u^2\in\mathcal{W}$, therefore \eqref{systunq2} immediately yields $\u^1(t)=\u^2(t)$ for all $t\in[0,T]$. 
\vskip 2mm
\noindent
\textbf{Case II:} \emph{For $r>3$ in $d=3$}. By using H\"older's and Young's inequalities, we find
\begin{align}\label{b1b2df2}
|\langle \B(\u^1)-\B(\u^2),\w\rangle|=|\langle \B(\w, \u^2), \w \rangle|
\leq\mu\|\w\|_{\V}^2+\frac{1}{4\mu}\|\u^2\w\|_{\H}^2.
\end{align}
We estimate the last term $\|\u^2\w\|_{\H}^2$ by using H\"older's and Young's inequalities with exponent $\frac{r-1}{2}$ and $\frac{r-1}{r-3}$ 
\begin{align}\label{b1b2df3}
	  \|\u^2\w\|_{\H}^2\leq\frac{\mu\beta}{2}\| |\u^2 |^{\frac{r-1}{2}}\w\|_{\H}^2 
	 +4\mu\varrho_3\|\w\|_{\H}^2,
\end{align}
where $\varrho_3= \frac{1}{4 \mu} \left(\frac{r - 3}{r - 1}\right)
\left( \frac{4}{\mu \beta (r - 1)} \right)^{\frac{2}{r -3}}$. On combining \eqref{b1b2df2}-\eqref{b1b2df3} into \eqref{systunq}, we get
\begin{align*}
&\|\w(t)\|_{\H}^2+
\frac{\beta}{2}\int_0^t\||\u^1(s)|^{\frac{r-1}{2}}\w(s)\|^2_{\H}\d s
+\frac{\beta}{4}\int_0^t\||\u^2(s)|^{\frac{r-1}{2}}\w(s)\|^2_{\H}\d s
\nonumber\\&\leq
\|\w(0)\|_{\H}^2+2(\varrho_1 +\varrho_2 + \varrho_3+\mathcal{K}) \int_0^t\|\w(s)\|_{\H}^2\d s, \text{ for all } \ t\in[0,T].
\end{align*}
An application of Gr\"onwall's Lemma immediately implies that $\u^1(t)=\u^2(t)$ for all $t\in[0,T]$. 
\vskip 0.2cm
\noindent
\textbf{Case III:} \emph{For $d=r=3$ with $2 \beta \mu >1$}. 
Modifying the calculations \eqref{4.766}, we compute
\begin{align}\label{Ct2}
|\alpha| |\langle\wi{\C}(\u^1)-\wi{\C}(\u^2),\w\rangle|
&\leq
2^{q-2}q|\alpha|\int_{\mathfrak{D}} (|\u^1(x)|^{q-1}+|\u^2(x)|^{q-1})|\w(x)|^2 \d x
\nonumber\\&\leq
\frac{1}{4\mu}\||\u^1|\w\|_{\H}^2+\frac14\left(\beta-\frac{1}{2\mu}\right)
\||\u^2|\w\|_{\H}^2+(\varrho_4+\varrho_5)\|\w\|_{\H}^2,
\end{align}
where $\varrho_4:=\left(\frac{3-q}{2}\right)\left(2^{q-1}q|\alpha|\mu(q-1)\right)^{\frac{q-1}{3-q}}
$ and 
$\varrho_5:=\left(\frac{3-q}{2}\right)\left(\frac{2^{q-1}q|\alpha|(q-1)}{\left(\beta-\frac{1}{2\mu}\right)}
\right)^{\frac{q-1}{3-q}}$. On substituting \eqref{Ct2} along with \eqref{4.76} (for $r=3$), \eqref{b1b2df2} and \eqref{4.81} into \eqref{eqn-unique-non}, and simplifying, we further obtain
\begin{align*}
	&\|\w(t)\|_{\H}^2+
	\left(\beta-\frac{1}{2\mu}\right)\int_0^t\||\u^1(s)|\w(s)\|^2_{\H}\d s
	+\frac12\left(\beta-\frac{1}{2\mu}\right)\int_0^t\||\u^2(s)|\w(s)\|^2_{\H}\d s
	\nonumber\\&\leq
	\|\w(0)\|_{\H}^2+2(\varrho_4+\varrho_5+\mathcal{K})\int_0^t\|\w(s)\|_{\H}^2,
\end{align*}
for all $t\in[0,T]$. Finally, for $2\beta\mu > 1$, Gr\"onwall's Lemma yields uniqueness and this completes the proof.
\end{proof}
 Now, let us give the example of a zig–zag function which are commonly used for adhesive or contact laws and satisfies  Hypothesis \ref{hyptheta} and the condition \eqref{theta-Cond}. For more information, one can see \cite[Application 8.2.6]{DGDM}.
\begin{example}
Consider the function $\theta_i : \mathfrak{D} \times (0,T) \times \mathbb{R} \to \mathbb{R}$  defined by 
 $$\theta_i(x,t,\xi):=	\beta_i(\xi),$$
where $\beta_i:\mathbb{R}\to \mathbb{R}$ is given by
	\[
	\beta_i(\xi)=
	\begin{cases}
		-3, & \xi<-2 \\[4pt]
		2\xi+1, & -2\leq \xi<-1,\\[4pt]
		-\xi, & -1\leq \xi<0,\\[4pt]
		3\xi, & 0\leq \xi<1,\\[4pt]
		-\xi+4, & 1\leq \xi<2,\\[4pt]
		2, & \xi\geq 2.
	\end{cases}
	\]
Now, let us verify  that the function $\theta_i$ satisfies Hypothesis \ref{hyptheta} and the condition \eqref{theta-Cond}.

\textbf{(i) Local boundedness.}
		Since $\beta_i$ is piecewise linear and constant outside the interval $[-2,2]$, it is bounded on every bounded interval of $\mathbb{R}$. Hence, for every $r>0,$ there exists a constant $c(r)>0$ such that
		$$ |\theta_i(x,t,\xi)|=|\beta_i(\xi)|\leq c(r)$$
		for all $(x,t)\in \mathfrak{D}\times(0,T)$ and $|\xi|\le r$.
		
		\textbf{(ii) Continuity in $(x,t)$.}
 Moreover, the function $\theta_i(x,t,\xi)=\beta_i(\xi)$ is independent of $(x,t)$. Therefore, for every $\xi\in\mathbb{R}$, the mapping
		$$ (x,t)\mapsto \theta_i(x,t,\xi)$$
		is continuous on $\mathfrak{D}\times(0,T)$.
		
\textbf{(iii) Growth condition.}
		Since $\beta_i$ is piecewise linear, there exist constants $C_{1,i},C_{2,i}>0$ such that
		$$
		|\theta_i(x,t,\xi)|=|\beta_i(\xi)|\le C_{1,i}+C_{2,i}|\xi|,
		$$
		for all $(x,t,\xi)\in\mathfrak{D}\times(0,T)\times\mathbb{R}$. Hence the growth condition
		$$
		|\theta_i(x,t,\xi)|\le \bar{\alpha}(x,t)+C_{1,i}+C_{2,i}|\xi|
		$$
		holds with $\bar{\alpha}(x,t)=0$.
		
		\textbf{(iv) Ultimate monotonicity in $\xi$.}
		From the definition of $\beta_i$, we observe that
		$$
		\beta_i(\xi)=-3\le 0 \text{ for } \xi<-2,
		$$
		and
		$$
		\beta_i(\xi)=2\ge 0 \text{ for } \xi\ge 2.
		$$
		Thus there exists $\phi_i=2$ such that
		$$
	 \esssup\limits_{\xi\in(-\infty,-\phi_i]}\theta_i(x,t,\xi)\le0
		\quad \text{and} \quad
		0\le
	 \essinf\limits_{\xi\in[\phi_i,\infty)}\theta_i(x,t,\xi).
		$$
Also, the function satisfies the condition \eqref{theta-Cond} with $\mathcal{K}=1$.
%		Therefore, the graph $(\xi,\theta_i(x,t,\xi))$ is ultimately increasing with respect to $\xi$.
	\end{example}

\medskip\noindent
\textbf{Acknowledgments:} The first author gratefully acknowledges the Ministry of Education, Government of India (Prime Minister Research Fellowship, PMRF ID: 2803609), for financial support to carry out her research work, and second author would like to thank Ministry of Education, Government of India - MHRD for financial assistance.  Support for M. T. Mohan's research received from the National Board of Higher Mathematics (NBHM), Department of Atomic Energy, Government of India (Project No. 02011/13/2025/NBHM(R.P)/R\&D II/1137).

	\medskip\noindent	{\bf  Declarations:} 
	
	\noindent 	{\bf  Ethical Approval:}   Not applicable 
	
	\noindent  {\bf   Competing interests: } The authors declare that they have no competing interests. 
	
		\noindent  {\bf   Author Contributions: } All authors contributed equally.

%	\noindent 	{\bf   Funding: } NBHM, India: No. 02011/13/2025/NBHM(R.P)/R\&D II/1137 (M. T. Mohan).

	\noindent 	{\bf   Availability of data and materials: } Not applicable.


\begin{thebibliography}	{99}	


	
\bibitem{SNAHB} S. N. Antontsev and H. B. de Oliveira, The Navier-Stokes problem modified by an absorption term, Appl. Anal., {\bf 89}(12) (2010), 1805--1825.

 \bibitem{WAMTM} W. Akram and M. T. Mohan, Mixed Finite Element Method for a Hemivariational Inequality of Stationary convective Brinkman-Forchheimer Extended Darcy equations,
Submitted, (2025), \url{https://arxiv.org/abs/2508.02797}. 

\bibitem{WAMTM1} W. Akram and M. T. Mohan, Optimal control of a hemivariational inequality of stationary convective Brinkman-Forchheimer extended Darcy equations with numerical approximation, Comput. Methods Appl. Mech. Engrg., {\bf 452} (2026).


%\bibitem{JCNVT}  J. Cen, Nguyen, V. T. and Vetro, C. and Zeng, S., Weak solutions to the generalized Navier-Stokes equations with mixed boundary conditions and implicit obstacle constraints, Nonlinear Anal. Real World Appl. {\bf 73} (2023), Paper No. 103904, 18 pp.

\bibitem{PCAM}  P. Cherrier and A. Milani, Linear and quasi-linear evolution equations in Hilbert spaces, Grad. Stud. Math., \textbf{135},
American Mathematical Society, Providence, RI, 2012.

\bibitem{KCC} K. C. Chang, 
Variational methods for nondifferentiable functionals and their applications to partial differential equations, J. Math. Anal. Appl., \textbf{80}(1) (1981),  102--129.


\bibitem{ZCQJ} Z. Cai and Q. Jiu, Weak and strong solutions for the incompressible Navier-Stokes equations with damping, J. Math. Anal. Appl., {\bf 343}(2) (2008), 799--809.

\bibitem{VVMI}  V. V. Chepyzhov and M. I. Vishik, Attractors for Equations of Mathematical Physics, Amer. Math. Soc. Colloq. Publ., \textbf{49}, American Mathematical Society, Providence, RI, 2002.

\bibitem{RA} E. DiBenedetto, Real Analysis, Birkh\"auser Boston, Inc., Boston, MA, 2002.

  
\bibitem{LCE} L. C. Evans, Partial Differential Equations, Grad. Stud. Math., \textbf{19},
American Mathematical Society, Providence, RI, 2010.

\bibitem{Fang2016} C. Fang, W. Han, S. Migórski, and M. Sofonea, A class of hemivariational inequalities for nonstationary Navier-Stokes equations, Nonlinear Anal. Real World Appl., {\bf 31} (2016), 257--276. 

\bibitem{FKS} R. Farwig, H. Kozono and H. Sohr,
An $L^q$-approach to Stokes and Navier-Stokes equations in general domains,
Acta Math., \textbf{195} (2005), 21--53.


\bibitem{FHR} C. L. Fefferman, K. W. Hajduk and J. C. Robinson, Simultaneous approximation in Lebesgue and Sobolev norms via eigenspaces, Proc. Lond. Math. Soc. (3), \textbf{125}(4) (2022), 759--777.

 \bibitem{GGP}  G. P. Galdi, An Introduction to the Navier-Stokes Initial-Boundary Value Problem, in Fundamental Directions in Mathematical Fluid Mechanics, 1--70, Adv. Math. Fluid Mech. Birkh\"auser, Basel 2000.

\bibitem{SGMTM} S. Gautam and M. T. Mohan, On the convective Brinkman-Forchheimer equations, Dyn. Partial Differ. Equ., {\bf 22}(3) (2025),  191--233. 


\bibitem{SGKKMTM}  S. Gautam, K. Kinra and M.~T. Mohan, Feedback stabilization of convective Brinkman-Forchheimer extended Darcy equations, Appl. Math. Optim., {\bf 91}(1) (2025), 75.


\bibitem{DGDM} D. Goeleven and D. Motreanu, Variational and Hemivariational Inequalities: Theory, Methods and Applications, Vol. II: Unilateral Problems, Non convex Optim. Appl., {\bf 70}, Kluwer Academic Publishers, Boston, MA, (2003).

\bibitem{KWH} K. W. Hajduk and J. C. Robinson, Energy equality for the 3D critical convective Brinkman-Forchheimer equations, J. Differential Equations, \textbf{263}(11) (2017), 7141--7161.
   
\bibitem{CHL}  C. Heil, Introduction to Real Analysis, Grad. Texts in Math., \textbf{280},
Springer, Cham, 2019. 

\bibitem{WH1} W. Han, Variational-hemivariational inequalities: A brief survey on mathematical theory and numerical analysis, \url{arXiv:2512.10204}.

\bibitem{MHHQLM} W. Han, H. Qiu and L. Mei, On a Stokes hemivariational inequality for incompressible fluid flows with damping, Nonlinear Anal. Real World Appl., {\bf 79} (2024), 104131.

\bibitem{WHYYSZ}  W. Han, Y. Yao and S. Zeng, Well-posedness and numerical analysis of a nonstationary Stokes hemivariational inequality, Math. Mech. Solids {\bf 31} (2026), no.~3, 462--492

\bibitem{JHMMPD} J. Haslinger, M. Miettinen and P. D. Panagiotopoulos, Finite Element Method for Hemivariational Inequalities: Theory, Methods and Applications, Nonconvex Optim. Appl., \textbf{35}, Kluwer Academic Publishers, Dordrecht, Boston, London, 1999.


\bibitem{JSMT} J. Jindal, S. Gautam and M. T. Mohan, Well-posedness of a boundary hemivariational inequality for stationary and non-stationary 2D and 3D convective Brinkman-Forchheimer equations, Submitted, (2025), \url{https://arxiv.org/abs/2508.17093}.

\bibitem{KT2} V. K. Kalantarov and S. Zelik, Smooth attractors for the Brinkman-Forchheimer equations with fast growing nonlinearities, Commun. Pure Appl. Anal., \textbf{11}(5) (2012), 2037--2054.

\bibitem{PKa}  P. Kalita, Convergence of Rothe scheme for hemivariational inequalities of parabolic type, Int. J. Numer. Anal. Model., {\bf 10}(2) (2013), 445--465.

\bibitem{JLEM}  J.-L. Lions and E. Magenes, Non-Homogeneous Boundary Value Problems and Applications: Vol I, Springer-Verlag, New York-Heidelberg, 1972. 	

\bibitem{JLL}  J. L. Lions, Quelques m\'ethodes de r\'esolution des problèmes aux limites non lin\'eaires, Dunod, Paris, 1969.

 
\bibitem{JL} J. Leray, Sur le mouvement d'un liquide visqueux emplissant l'espace, Acta Math., {\bf 63}(1) (1934), 193--248.

\bibitem{SMAO} S. Mig\'orski and A. Ochal, Optimal control of parabolic hemivariational inequalities, J. Global Optim., \textbf{17}(1-4) (2000), 285--300.

\bibitem{SMAON} S. Mig\'orski and A. Ochal, Navier-Stokes problems modeled by evolution hemivariational inequalities, Discrete Contin. Dyn. Syst.,  Dynamical systems and differential equations. Proceedings of the 6th AIMS International Conference, 731--740, 2007. 

\bibitem{SMAOMS}  S. Mig\'orski, A. Ochal and M. Sofonea, Nonlinear Inclusions and Hemivariational Inequalities: Models and Analysis of Contact Problems, Advances in Mechanics and Mathematics, \textbf{26}, Springer, New York, 2013.

 \bibitem{MTT} P. A. Markowich, E. S. Titi and S. Trabelsi, Continuous data assimilation for the three dimensional Brinkman-Forchheimer-extended Darcy model, Nonlinearity, \textbf{29}(4) (2016), 1292--1328. 


\bibitem{MTMS} M. T. Mohan, Well-posedness and asymptotic behavior of stochastic convective Brinkman-Forchheimer equations perturbed by pure jump noise, Stoc PDE: Anal. Comp., {\bf 10}(2) (2022), 614--690.	

\bibitem{MTM} M. T. Mohan, Well-posedness of a  
stationary 2D and 3D convective Brinkman-Forchheimer extended darcy hemivariational inequality, J. Optim. Theory Appl., {\bf 208}, Ar. No. 118, 2026.

\bibitem{PDP} P. D. Panagiotopoulos, Inequality Problems in Mechanics and Applications, Convex and Nonconvex Energy Functions, Birk\"auser Boston, Inc., Boston, MA, 1985.

\bibitem{PDPHI} P. D. Panagiotopoulos, Hemivariational Inequalities, Applications in Mechanics and Engineering, Springer-Verlag, Berlin, 1993.

     
   \bibitem{LSM}  J. C. Robinson and  W. Sadowski, A local smoothness criterion for solutions of the 3D Navier-Stokes equations, Rend. Semin. Mat. Univ. Padova, \textbf{131} (2014), 159--178.
   

  \bibitem{MRXZ}	M. R\"ockner and X. Zhang, Tamed 3D Navier-Stokes equation: existence, uniqueness and regularity, Infin. Dimens. Anal. Quantum Probab. Relat. Top., {\bf 12}(4) (2009), 525--549.
     

\bibitem{HS} H. Sohr, The Navier-Stokes Equations, An Elementary Functional Analytic Approach, Modern Birkh\"auser Classics, Birkh\"auser/Springer Basel AG, Basel, 2001.

\bibitem{JSi}  J. Simon, Compact sets in the space $L^p(0,T;B)$, Ann. Mat. Pura Appl., {\bf 146} (1987) 65--96.
      
   \bibitem{TRo}  T. Roubicek, Nonlinear Partial Differential Equations with Applications, Internat. Ser. Numer. Math., \textbf{153}, Birkh\"auser Verlag, Basel, Boston, Berlin, 2005.
              
      \bibitem{Te} R. Temam, Navier-Stokes Equations: Theory and Numerical Analysis, Third edition, Stud. Math. Appl., \textbf{2}, North-Holland, Amsterdam, 1984.
    
 \bibitem{WWXC}  W. Wang, X. Cheng and W. Han, Analysis and finite element solution of a Navier-Stokes hemivariational inequality for incompressible fluid flows with damping, Nonlinear Anal. Real World Appl., {\bf 87} (2026), 16. 


\bibitem{EZ} E. Zeidler, Nonlinear Functional Analysis and its Applications, II/B: Nonlinear Monotone Operators, Springer-Verlag, 1990.
   
   \bibitem{ZZXW}	Z. Zhang, X. Wu and M. Lu, On the uniqueness of strong solution to the incompressible Navier-Stokes equations with damping, J. Math. Anal. Appl., {\bf 377}(1) (2011), 414--419.
   
 
\end{thebibliography}
\end{document}